\documentclass[11pt]{amsart}

\usepackage{
	amsmath,
 	amsfonts,
  	amssymb,
 	amsthm,
        datetime,
	}

\usepackage{tikz-cd}
\usepackage{tikz}

\usepackage{color}
\usepackage{array}
\usepackage[english]{babel}
\usepackage[utf8]{inputenc}

\usepackage[cal=boondoxo]{mathalfa}
\usepackage{mathtools}
\usepackage{bbm}

\usepackage[bookmarksdepth=3]{hyperref}
\usepackage[capitalise]{cleveref}

\usepackage{a4wide}

\theoremstyle{plain}
\newtheorem{theorem}{Theorem}[section]
\newtheorem{corollary}[theorem]{Corollary}
\newtheorem{lemma}[theorem]{Lemma}
\newtheorem{proposition}[theorem]{Proposition}
\newtheorem{conjecture}[theorem]{Conjecture}
\newtheorem{question}[theorem]{Question}

\theoremstyle{definition}
\newtheorem{remark}[theorem]{Remark}

\newtheorem*{theoremm}{Theorem}

\theoremstyle{definition}
\newtheorem{definition}[theorem]{Definition}

\def\makeautorefname#1#2{\expandafter\def\csname#1autorefname\endcsname{#2}}
\makeautorefname{lemma}{Lemma}%
\makeautorefname{proposition}{Proposition}%
\makeautorefname{remark}{Remark}%
\makeautorefname{section}{Section}%
\makeautorefname{question}{Question}%

\definecolor{shadecolor}{gray}{0.90}
\def\boitegrise#1#2{\medskip\begin{centerline}{\fcolorbox{black}{shadecolor}{
    \begin{minipage}[t]{#2}{\vphantom{~}#1\vphantom{$A_{\displaystyle{A_A}}$}}
            \end{minipage}~}}\end{centerline}\medskip}

\DeclareMathOperator{\id}{id}
\DeclareMathOperator{\Hom}{Hom}
\DeclareMathOperator{\End}{End}

\DeclareMathOperator{\Rep}{Rep}

\newcommand{\gl}{\mathfrak{gl}}
\newcommand{\qbinom}[2]{\genfrac{[}{]}{0pt}{}{#1}{#2}}
\newcommand{\moddet}[1]{\det\nolimits_{#1}}
\renewcommand{\tilde}[1]{\widetilde{#1}}

\DeclareMathOperator{\abs}{abs}
\DeclareMathOperator{\Gr}{Gr}
\DeclareMathOperator{\Irr}{Irr}
\DeclareMathOperator{\ev}{ev}
\DeclareMathOperator{\coev}{coev}

\DeclareMathOperator{\Tr}{Tr}
\DeclareMathOperator{\op}{op}
\DeclareMathOperator{\sym}{sym}
\DeclareMathOperator{\sh}{sh}
\DeclareMathOperator{\elem}{elem}
\DeclareMathOperator{\sVect}{sVect}
\DeclareMathOperator{\Fr}{Fr}

\def\arxivprefix{https://arxiv.org}

\definecolor{linkcolor}{rgb}{0,0,1} 
\hypersetup{
	colorlinks=true, 
	pdfstartview=FitV, 
	urlcolor=linkcolor, 
	linkcolor= linkcolor, 
	citecolor=linkcolor 
}

\title{Fourier matrices for $G(d,1,n)$ from quantum general linear groups}

\author{Abel Lacabanne}
\address{Institut de Recherche en Math\'ematique et Physique\\
Universit\'e catholique de Louvain\\ 
Chemin du Cyclotron 2\\ 
1348 Louvain-la-Neuve\\ 
Belgium}
\thanks{A.L. is a Postdoctoral Researcher of the Fonds de la Recherche Scientifique-FNRS}
\email{abel.lacabanne@uclouvain.be}

\begin{document}


\begin{abstract}
  We construct a categorification of the modular data associated with every family of unipotent characters of the spetsial complex reflection group $G(d,1,n)$. The construction of the category follows the decomposition of the Fourier matrix as a Kronecker tensor product of exterior powers of the character table $S$ of the cyclic group of order $d$. The representation of the quantum universal enveloping algebra of the general linear Lie algebra $\mathfrak{gl}_m$, with quantum parameter an even root of unity of order $2d$, provides a categorical interpretation of the matrix $\bigwedge^m S$. We also prove some positivity conjectures of Cuntz at the decategorified level.
\end{abstract}



\maketitle




In the theory of representations of finite groups of Lie type, unipotents characters are an important object of study. Indeed, these characters are the building blocks for the irreducible characters. They have been classified by Lusztig and are partitioned into families. To each family of unipotent characters, Lusztig has associated a modular datum consisting of a Fourier matrix and of the eigenvalues of the Frobenius.

One important observation is that the classification of unipotents characters does not depend on the finite field of definition of the reductive group, and only depends on the structure of the Weyl group $W$, neither do the modular data. It has been later realized by Lusztig \cite{lusztig-unipotent-coxeter}, that one can define a similar set of ``unipotent characters'' for a finite Coxeter group which is not a Weyl group, together with polynomials which share similar properties with the degrees of unipotent characters of a a finite group of Lie type. Thereafter, modular data associated with families of unipotent characters of finite Coxeter groups have also been constructed \cite{lusztig-exotic}.

As Weyl groups are rational reflection groups and Coxeter groups are real reflection groups, similar combinatorial version of unipotent characters for a certain class of complex reflection groups, called ``spetsial groups'', have been studied by Broué, Malle and Michel \cite{spetsesI,spetsesII}. Malle \cite{unipotente} has tackled the case of spetsial imprimitive complex reflection groups, by defining unipotent characters and their degrees, families of characters and a modular data for each family. The combinatorics developed by Malle is a generalization of Lusztig's combinatorics for the classification of unipotent characters of a group of Lie type with Weyl group of type $B$.

\medskip

In this article, we will study the case of the complex reflection group $G(d,1,n)$ and the problem of categorifying the modular data associated to the families of unipotent characters. Indeed, modular categories are known to produce modular data, and given a modular datum, it is a classical problem to determine whether it arises from a category or not. As the modular data associated with a family of unipotent characters of $G(d,1,n)$ does not satisfy a positivity property, one needs to use triangulated or super categories.

A first step in this direction has been achieved by Bonnafé and Rouquier \cite{bonnafe-rouquier}: they gave a categorical interpretation of the modular data associated with the unique non trivial family of unipotent characters of the cyclic group $G(d,1,1)$. Their category is constructed from the Drinfeld double of the Taft algebra, which is a finite dimensional version of the quantum enveloping algebra of the standard Borel of $\mathfrak{sl}_2$. In \cite{slightly-deg}, the author explained how to reinterpret the category of Bonnafé and Rouquier into the framework of slightly degenerate categories. This framework turned out to be well adapted for the problem of categorifying modular data: the modular data of some families of the complex reflection group $G(d,1,n)$ arise from the representation of the Drinfeld double of the quantum enveloping algebra of the standard Borel of $\mathfrak{sl}_m$ \cite{drinfeld-double-modular-data}.

The aim of this paper is to give a categorical interpretation of every modular data arising from a family of unipotent characters of $G(d,1,n)$. More precisely, Cuntz \cite{cuntz-fusion} has noticed that the Fourier matrix of any family of unipotent characters of $G(d,1,n)$ is obtained from the Kronecker tensor product of some elementary building blocks of the form $\bigwedge^m S$ with $S$ being the renormalized character table of the cyclic group of order $d$. This matrix $\bigwedge^m S$, together with a diagonal matrix $\bigwedge^mT$ define a modular datum. Therefore, we will first construct a categorification of these building blocks which will then give a categorification of the Fourier matrix of any family of unipotent characters of $G(d,1,n)$.

The categorical interpretation of $\bigwedge^m S$ will be obtained using representations of the quantum group $\mathcal{U}_{\xi}(\mathfrak{gl}_m)$ where $\xi$ is an even root of unity of order $2d$. More precisely, we consider the semisimplification of the category of tilting modules, as in the classical construction for simple Lie algebra. But this will produce a category with an infinite number of non-isomorphic simple objects. Fortunately, many of the invertible objects of this category lie in the symmetric center and one can modularize this category in order to obtain a category that we denote by $\tilde{\mathcal{D}}_{m,\xi}$. The category $\tilde{\mathcal{D}}_{m,\xi}$ admits many pivotal structure and the choice of a pivotal structure is related to the choice of the unit element in the fusion algebra defined by $\bigwedge^m S$.

\begin{theoremm}[\cref{thm:categorification}]
  The category $\tilde{\mathcal{D}}_{m,\xi}$ is a categorification of the modular datum defined by $\bigwedge^m S$ and $\bigwedge^m T$.
\end{theoremm}

Each family $\mathcal{F}$ of unipotent characters of $G(d,1,n)$ is defined by the choice of non-negative integers $w_1,\ldots,w_r$  and integers $0<n_1,\ldots,n_r \leq d$. Up to some constant, the Fourier matrix of the corresponding family of unipotent characters is a submatrix of the complex conjugate of $\bigwedge^{n_1} S \otimes \cdots \otimes \bigwedge^{n_r} S$, the product being the Kronecker tensor product of matrices; there is a similar construction for the eigenvalues of the Frobenius using the various matrices $\bigwedge^{n_i} T$. It is then natural to consider the Deligne tensor product $\tilde{\mathcal{D}}_{\underline{n},\xi^{-1}} = \tilde{\mathcal{D}}_{n_{1},\xi^{-1}}\boxtimes \cdots \boxtimes \tilde{\mathcal{D}}_{n_{r},\xi^{-1}}$.

\begin{theoremm}[\cref{thm:categorification_d1n}]
  There exists a non-degenerate subcategory $\tilde{\mathcal{E}}_{\underline{n},\xi}$ of  $\tilde{\mathcal{D}}_{\underline{n},\xi^{-1}}$ which is a categorification of the modular datum associated to the family of unipotent characters $\mathcal{F}$.
\end{theoremm}

\medskip

The first section of this paper is a recollection of known results concerning fusion algebras, modular data and their categorification. Then, in a second section, we study thoroughly the representations of the quantum group $\mathcal{U}_q(\mathfrak{gl}_m)$ at an even root of unity. The semisimplification of the category of tilting modules gives rises to a semisimple braided pivotal category, with a non-trivial symmetric center. We then explain the effect of killing the symmetric center and describe the modular datum that one can extract from the category $\tilde{\mathcal{D}}_{m,\xi}$. The third part is dedicated to the comparison of this modular datum with the $m$-th exterior power of the renormalized character table of the cyclic group of order $d$. Using the categorical interpretation of the matrix $\bigwedge^{m} S$, we prove several conjectures of Cuntz related to positivity questions. Finally, in the fourth section, we consider the modular data defined by Malle, which are associated to families of unipotent characters of the complex reflection group $G(d,1,n)$. These modular data are shown to be obtained from a variant of the categorical construction of the second section.




\section{Categorical prolegomena}
\label{sec:cat_prol}

Our base field is the field of complex numbers $\mathbb{C}$, but most of the materials of this section remains true over an algebraically closed field of characteristic $0$.

\subsection{Fusion algebras from $S$-matrices}
\label{sec:fusion_alg}

We start by recollection a some basic facts on the notion of a fusion algebra.

Let $I$ be a finite set and $\mathbb{S}$ a square matrix with complex entries indexed by $I$. We suppose that $\mathbb{S}$ is symmetric and unitary, and that there exists $i_0\in I$ such that $\mathbb{S}_{i_0,i}\neq 0$ for every $i\in I$. We also suppose that for every $i,j,k\in I$, the number
    \begin{equation}
      N_{i,j}^k = \sum_{l\in I}\frac{\mathbb{S}_{i,l}\mathbb{S}_{j,l}\overline{\mathbb{S}_{k,l}}}{\mathbb{S}_{i_0,l}} \in \mathbb{Z}.\label{eq:verlinde}
    \end{equation}
is an integer.

To such a matrix $\mathbb{S}$, we associate a $\mathbb{Z}$-algebra $A_{\mathbb{S}}$, which is free as a $\mathbb{Z}$-module with basis $(b_i)_{i\in I}$. The product is defined on the basis by
\[
  b_i\cdot b_j = \sum_{k\in I}N_{i,j}^k b_k
\]
and is linearly extended to $A_{\mathbb{S}} = \bigoplus_{i\in I}\mathbb{Z} b_i$. It is easily checked that this multiplication is associative and that $b_{i_0}$ is the unit element. The algebra $A_{\mathbb{S}}$ is the \emph{fusion algebra} associated with the matrix $\mathbb{S}$ and the integers $(N_{i,j}^k)_{i,j,k\in I}$ are the structure constants with respect to the basis $(b_i)_{i\in I}$. Note that multiplying $\mathbb{S}$ by any complex number $\omega$ of module $1$ leads to an isomorphic fusion algebra $A_{\omega \mathbb{S}}\simeq A_{\mathcal{S}}$.

\begin{lemma}
  \label{lem:change_signs}
  Let $\Sigma$ be a diagonal matrix with entries $(\sigma_i)_{i\in I}$ with $\sigma_i\in\{\pm 1\}$. Let $\mathbb{S}'$ be the matrix $\Sigma \mathbb{S} \Sigma^{-1}$. Denote by $(b_i)_{i \in I}$ the basis of $A_{\mathbb{S}}$ and by $(b_i')_{i\in I}$ the basis of $A_{\mathbb{S}'}$. Then $b_i \mapsto \sigma_{i_0}\sigma_i b'_i$ is an algebra isomorphism between $A_{\mathbb{S}}$ and $A_{\mathbb{S}'}$.
\end{lemma}

\begin{proof}
  If we denote by $(N_{i,j}^k)_{i,j,k\in I}$ (resp. $(N_{i,j}^{'k})_{i,j,k\in I}$) the structure constants of $A_{\mathbb{S}}$ (resp. $A'_{\mathbb{S}}$), we have that
  \[
    N_{i,j}^{'k} = \sigma_i\sigma_j\sigma_k\sigma_{i_0} N_{i,j}^k,
  \]
  for any $i,j,k\in I$. The lemma follows easily from this equality.
\end{proof}

Therefore, conjugation by a diagonal matrix of signs translates into a change of signs of the basis of the fusion algebra. Given a matrix $\mathbb{S}$ we are interested in the following question:

\begin{question}
  \label{qu:change_signs}
Does it exist a collection of signs such that the algebra $A_{\mathbb{S}'}$, obtained from $\mathbb{S}'$ as in \cref{lem:change_signs}, has non-negative structure constants?
\end{question}

It is usually not easy to give an answer to this question since the structure constants might be tedious to compute. Nevertheless, we will later give some examples of such matrices $\mathbb{S}$ and answer to this question via categorical methods.

\subsection{Modular data and fusion algebras}
\label{sec:modular_data}

We now define the notion of a modular datum, which is inspired from \cite{gannon,lusztig-exotic}.

\begin{definition}
  A \emph{modular datum} is a quadruple $(I,i_0,\mathbb{S},\mathbb{T})$, where $I$ is a finite set, $i_0$ is an element of $I$ called \emph{special} or \emph{distinguished}, $\mathbb{S}$ is a complex matrix with entries indexed by $I$, and $\mathbb{T}$ is a complex diagonal matrix with entries indexed by $I$ satisfying the following conditions:
  \begin{itemize}
  \item $\mathbb{S}$ is symmetric and unitary,
  \item $\mathbb{S}$ and $\mathbb{T}$ define a projective representation of $SL_2(\mathbb{Z})$: there exists $\xi \in \mathbb{C}^*$ such that
    \[
      \mathbb{S}^4=\id,\quad (\mathbb{ST})^3=\xi\id,\quad\text{and}\quad \mathbb{S}^2\mathbb{T}=\mathbb{T}\mathbb{S}^2,
    \]
  \item for all $i\in I$, $\mathbb{S}_{i_0,i}\neq 0$,
  \item for all $i,j,k\in I$, we have
    \[
      N_{i,j}^k = \sum_{l\in I}\frac{\mathbb{S}_{i,l}\mathbb{S}_{j,l}\overline{\mathbb{S}_{k,l}}}{\mathbb{S}_{i_0,l}} \in \mathbb{Z}.
    \]
  \end{itemize}
\end{definition}

Note that by renormalizing $\mathbb{T}$ by a third root of $\xi$, one can obtain a genuine representation of $SL_{2}(\mathbb{Z})$.

Since the matrix $\mathbb{S}$ of a modular datum $(I,i_0,\mathbb{S},\mathbb{T})$ satisfies the conditions of \cref{sec:fusion_alg}, we have at our disposal the fusion algebra $A_{\mathbb{S}}$. Even if this algebra depends only on $\mathbb{S}$, we will call it the fusion algebra associated with the modular datum $(I,i_0,\mathbb{S},\mathbb{T})$.

\subsection{Non-degenerate and slightly degenerate categories}
\label{sec:non-deg_sl-deg}

Using modular categories, one can try to categorify a modular datum whose fusion ring has non-negative structure constants. In \cite{slightly-deg}, the author explains how slightly degenerate pivotal fusion categories provides a broader framework for the categorifications of modular data where the fusion ring may have negative structure constants. We quickly recall these categorical notions, and the main results of \cite{slightly-deg}. 

\subsubsection{Pivotal fusion categories}

For the definition of a fusion category, we refer to \cite[Definition 4.1.1]{egno}. The tensor product will be denoted by $\otimes$, the unit object by $\mathbf{1}$, and the associativity and unit constraints will be omitted. The set of isomorphism classes of simple objects of a fusion category $\mathcal{C}$ is denoted by $\Irr(\mathcal{C})$ and its Grothendieck ring by $\Gr(\mathcal{C})$. The latter is a free abelian group with generators given by $([X])_{\in\Irr(\mathcal{C})}$ and the multiplication is given by the tensor product:
\[
  [X][Y] = \sum_{Z\in \Irr(\mathcal{C})}N_{X,Y}^Z[Z],
\]
where $N_{X,Y}^Z$ denotes the multiplicity of the simple object $Z$ in the tensor product $X\otimes Y$.

The \emph{left dual} (resp. \emph{right dual}) $(X^*,\ev_X,\coev_X)$ (resp. $({}^*X,\ev'_X,\coev'_X)$) of an object $X$ consists of the datum of an object $X^*$, a evaluation map $\ev_X\colon X^*\otimes X \rightarrow \mathbf{1}$ (resp. $\ev'_X\colon X\otimes {}^*X \rightarrow \mathbf{1}$) and a coevaluation map $\coev_X\colon\mathbf{1}\rightarrow X\otimes X^*$ (resp. $\coev'_X\colon\mathbf{1}\rightarrow {}^*X\otimes X$) satisfying
\begin{align*}
  (\id_X\otimes \ev_X)\circ (\coev_X\otimes \id_X) &= \id_X &\text{and}&&   (\ev_X\otimes\id_{X^*})\circ(\id_{X^*}\otimes\coev_{X})&=\id_{X^*}\\
  \intertext{(resp.}
  (\ev'_X\otimes \id_X)\circ (\id_X\otimes \coev'_X) &= \id_X &\text{and}&&   (\id_{X^*}\otimes\ev'_{X})\circ(\coev'_{X}\otimes\id_{X^*})&=\id_{X^*}).
\end{align*}
Left (resp. right) duals are unique up to unique isomorphism and a monoidal category is said to be rigid if every object admits a left and a right dual. Recall that, by definition, a fusion category is rigid.

If $X$ and $Y$ have left duals, we also have the notion of a \emph{left dual map} for any $f\in\Hom_{\mathcal{C}}(X,Y)$. It is a map $f^*\in\Hom_{\mathcal{C}}(Y^*,X^*)$ and is defined by
\[
  f^* = (\ev_{Y}\otimes \id_{X^*}) \circ (\id_{Y^*}\otimes f \otimes \id_{X^*}) \circ (\id_{Y^*}\otimes \coev_X).
\]

A rigid monoidal category is said to be \emph{pivotal} if there exists a natural isomorphism $a_X\colon X\rightarrow X^{**}$ compatible with the tensor product, that is $a_{X\otimes Y} = a_X \otimes a_Y$, up to the usual identification between $(X\otimes Y)^{**}$ and $X^{**}\otimes Y^{**}$.

In a pivotal fusion category we have at our disposal the \emph{right quantum trace} of an endomorphism. Given $f\in\End_{\mathcal{C}}(X)$, its right quantum trace is the unique scalar $\Tr(f)$ such that the composition
\[
  \begin{tikzcd}
    \mathbf{1}\ar[r,"\coev_X"] & X\otimes X^* \ar[r,"(a_X\circ f)\otimes \id_X"] &[3em] X^{**}\otimes X^* \ar[r,"\ev_{X^*}"] & \mathbf{1}
  \end{tikzcd}
\]
is equal to $\Tr(f)\id_{\mathbf{1}}$. The \emph{right quantum dimension} $\dim(X)$ of an object $X$ is simply the right quantum trace of the identity morphism. There also exists a notion of left quantum trace and left quantum dimension. A pivotal structure is said to be \emph{spherical} if $\dim(X) = \dim(X^*)$ for every object, or equivalently if the left and right quantum traces coincide.  Most of the pivotal structures we will consider in \cref{sec:qgl} are not spherical.

\begin{remark}
  We choose the above convention for right quantum traces which is different to \cite[Definition 4.7.1]{egno}. Our convention follows from graphical calculus, where the right quantum trace of an endomorphism is obtained by closing a diagram on the right.
\end{remark}

Simple objects of a fusion category have a non-zero quantum dimension and we define the categorical dimension of such a fusion category $\mathcal{C}$ by
\[
  \dim(\mathcal{C}) = \sum_{X\in \Irr(\mathcal{C})}\lvert \dim(X) \rvert^2.
\]
It is a positive real number.

\subsubsection{Braided categories, degeneracy and twist}

A \emph{braiding} on a monoidal category is the datum of a binatural isomorphism $c_{X,Y}\colon X\otimes Y \rightarrow Y\otimes X$ such that the hexagon axioms are satisfied:
\[
  c_{X\otimes Y,Z} = (c_{X,Z}\otimes \id_Y)\circ(\id_X\otimes c_{Y,Z})
  \quad \text{and} \quad
  c_{X,Y\otimes Z} = (\id_Y\otimes c_{X,Z})\circ(c_{X,Y}\otimes \id_Z).
\]
As an immediate consequence, the Grothendieck ring of a braided fusion category is commutative.

\medskip

A simple object $X$ of a braided category is said to be \emph{transparent} if $c_{Y,X}\otimes c_{X,Y} = \id_{X\otimes Y}$ for every object $Y$. The \emph{symmetric center} $\mathcal{Z}_{\sym}(\mathcal{C})$ of a braided category $\mathcal{C}$ is the full subcategory of $\mathcal{C}$ whose objects are the transparent objects of $\mathcal{C}$.

We say that a braided fusion category is \emph{non-degenerate} if its symmetric center is tensor generated by the unit object $\mathbf{1}$. We say that a braided fusion category is \emph{slightly degenerate} if its symmetric center is equivalent to the symmetric category $\sVect$ of finite dimensional super vector spaces, with braiding $c_{V,W}(v\otimes w) = (-1)^{\lvert v\rvert\lvert w \rvert} w\otimes v$ for any super vector spaces $V,W$, $v\in V$ and $w\in W$.

\medskip

We now define the $S$-matrix of a braided fusion category which will play a prominent role.

\begin{definition}
The \emph{$S$-matrix} of a braided pivotal fusion category $\mathcal{C}$ is the matrix $S=(S_{X,Y})_{X,Y\in\Irr(\mathcal{C})}$ indexed by $\Irr(\mathcal{C})$ with entries given by the left quantum trace of the double braiding:
\[
  S_{X,Y} = \Tr(c_{Y,X}\circ c_{X,Y}).
\]
\end{definition}

If a simple object $X$ is transparent, then for all $Y\in \Irr(\mathcal{C})$ one have $S_{X,Y} = \dim(X)\dim(Y)$. The converse is also true.

\begin{proposition}[{\cite[Proposition 8.20.5]{egno}}]
  Let $\mathcal{C}$ be a braided fusion category. Then an object $X\in Irr(\mathcal{C})$ is transparent if and only if for all $Y\in \Irr(\mathcal{C})$ one has $S_{X,Y} = \dim(X)\dim(Y)$.
\end{proposition}

\medskip

In a rigid braided category, there always exists an isomorphism $u_X\colon X\rightarrow X^{**}$, called the \emph{Drinfeld morphism} which is given by the following composition
\[
  \begin{tikzcd}
    X \ar[r,"\id_X\otimes \coev_{X^*}"] &[3em]  X\otimes X^* \otimes X^{**} \ar[r,"c_{X,X^*}\otimes \id_{X^{**}}"] &[3em] X^*\otimes X \otimes X^{**} \ar[r,"\ev_X\otimes \id_{X^{**}}"] &[3em] X^{**}. 
  \end{tikzcd}
\]
However, this morphism is not a pivotal structure, but satisfies
\[
  u_X\otimes u_Y = u_{X\otimes Y}\circ c_{Y,X} \circ c_{X,Y}.
\]

We now suppose that $\mathcal{C}$ is a braided pivotal fusion category. The composition of the Drinfeld morphism and of the pivotal structure give rise to an endofunctor $\theta=a\circ u^{-1}$ of the identity. This endofunctor is a \emph{twist}, that is satisfies
\[
  \theta_{X\otimes Y} = \theta_X\otimes \theta_Y \circ c_{Y,X} \circ c_{X,Y}
\]
and the pivotal structure $a$ is spherical if and only if $\theta_{X^*} = (\theta_X)^*$ for every object $X$, that is $a$ is spherical structure if and only if $\theta$ is a \emph{ribbon}.

\subsubsection{Non-degenerate and slightly degenerate categories}
\label{sec:slightly-deg}

Under some assumptions on the symmetric center of a braided pivotal fusion category, the $S$-matrix and the twist give rise to a modular datum, with fusion algebra related to the Grothendieck ring of the category.

\medskip

\boitegrise{\textbf{Hypothesis:} We assume that the category $\mathcal{C}$ is non-degenerate.}{0.8\textwidth}

There exists an invertible object $\bar{\mathbf{1}}$ such that $S_{\bar{\mathbf{1}},Y} = \dim(\bar{\mathbf{1}})\dim(Y^*)$ for every simple object $Y$, see \cite[\S 2.4]{slightly-deg}. Denote by $\mathbb{T}$ the diagonal matrix with entries $(\delta_{X,Y}\theta_X)_{X,Y\in \Irr(\mathcal{C})}$, where we have identified $\theta_X$ and the unique scalar $\lambda$ such that $\theta_X=\lambda\id_X$.

We define the renormalized matrix $\mathbb{S}$ as
\[
  \mathbb{S}=\frac{S}{\sqrt{\dim(\mathcal{C})}\sqrt{\dim(\bar{\mathbf{1}})}},
\]
where $\sqrt{\dim(\bar{\mathbf{1}})}$ is a square root of $\dim(\bar{\mathbf{1}})$.

\begin{proposition}[{\cite[Theorem 2.22]{slightly-deg}}]
  Let $\mathcal{C}$ be a non-degenerate braided pivotal fusion category. Then $(\Irr(\mathcal{C}),\mathbf{1},\mathbb{S},\mathbb{T})$ is a modular datum. The associated fusion algebra $A_{\mathbb{S}}$ is isomorphic to the Grothendieck ring of $\mathcal{C}$.
\end{proposition}

\medskip

\boitegrise{\textbf{Hypothesis:} We assume that the category $\mathcal{C}$ is slightly degenerate. We also suppose that the twist of the simple transparent non-unit object of $\mathcal{C}$ is of quantum dimension $-1$ and of twist $1$.}{0.8\textwidth}

Denote by $\varepsilon$ the unique simple transparent of $\mathcal{C}$ such that $\varepsilon\not\simeq \mathbf{1}$. Then $\varepsilon\otimes \varepsilon \simeq \mathbf{1}$, $\dim(\varepsilon)=-1$ and $\theta_\varepsilon=1$.

Tensoring by $\varepsilon$ has no fixed points on $\Irr(\mathcal{C})$ and then we choose a subset $J\subseteq \Irr(\mathcal{C})$ containing one element for each orbit of simple object under tensorization by $\varepsilon$. We will make the assumption that $\mathbf{1}\in J$ (hence $\varepsilon \not\in J$). There again exists an invertible object $\bar{\mathbf{1}}\in J$ such that $S_{\bar{\mathbf{1}},Y} = \dim(\bar{\mathbf{1}})\dim(Y^*)$ for every simple object $Y$. We consider the submatrix $\tilde{S}$ of $S$ whose entries are indexed by $J$. Since $\Irr(\mathcal{C}) = J \sqcup J\otimes \varepsilon$, we have
\[
  S =
  \begin{pmatrix}
    \tilde{S} & -\tilde{S}\\
    -\tilde{S} & \tilde{S}
  \end{pmatrix}.
\]
Denote by $\tilde{\mathbb{T}}$ the diagonal matrix with entries $(\delta_{X,Y}\theta_X)_{X,Y\in J}$, where we have once again identified $\theta_X$ and the unique scalar $\lambda$ such that $\theta_X=\lambda\id_X$.

We define the renormalized matrix $\tilde{\mathbb{S}}$ as
\[
  \tilde{\mathbb{S}}=\frac{\tilde{S}}{\sqrt{\frac{1}{2}\dim(\mathcal{C})}\sqrt{\dim(\bar{\mathbf{1}})}},
\]
where $\sqrt{\dim(\bar{\mathbf{1}})}$ is a square root of $\dim(\bar{\mathbf{1}})$.

\begin{proposition}[{\cite[Theorem 3.7]{slightly-deg}}]
  Let $\mathcal{C}$ be a slightly degenerate braided pivotal fusion category. Then $(J,\mathbf{1},\tilde{\mathbb{S}},\tilde{\mathbb{T}})$ is a modular datum. The associated fusion algebra $A_{\tilde{\mathbb{S}}}$ is isomorphic to the quotient $\Gr(\mathcal{C})/([\varepsilon]+[\mathbf{1}])$ of the Grothendieck ring of $\mathcal{C}$.
\end{proposition}

The two quotients $\Gr(\mathcal{C})/([\varepsilon]+[\mathbf{1}])$ and $\Gr(\mathcal{C})/([\varepsilon]-[\mathbf{1}])$ have a basis indexed by $J$ and are both quotients of $\Gr(\mathcal{C})$:
\[
  \begin{tikzcd}
    & \Gr(\mathcal{C}) \ar[dl,two heads,"{[\varepsilon]=-[\mathbf{1}]}"'] \ar[dr,two heads,"{[\varepsilon]=[\mathbf{1}]}"] &\\
    \Gr(\mathcal{C})/([\varepsilon]+[\mathbf{1}]) & & \Gr(\mathcal{C})/([\varepsilon]-[\mathbf{1}])
  \end{tikzcd}
\]
One can easily describe their structure constants using the structure constants of $\Gr(\mathcal{C})$: for $X,Y,Z\in J$, the structure constant of $\Gr(\mathcal{C})/([\varepsilon]+[\mathbf{1}])$ are given by $N_{X,Y}^Z-N_{X,Y}^{Z\otimes \varepsilon}$ and the structure constant of $\Gr(\mathcal{C})/([\varepsilon]-[\mathbf{1}])$ are given by $N_{X,Y}^Z+N_{X,Y}^{Z\otimes \varepsilon}$. Hence, if $N_{X,Y}^ZN_{X,Y}^{\varepsilon\otimes Z}=0$ for every $X,Y,Z\in \Irr(\mathcal{C})$, the structure constants of $\Gr(\mathcal{C})/([\varepsilon]-[\mathbf{1}])$ are the absolute values of the structure constant of $\Gr(\mathcal{C})/([\varepsilon]+[\mathbf{1}])$. The condition $N_{X,Y}^ZN_{X,Y}^{\varepsilon\otimes Z}=0$ for every $X,Y,Z\in \Irr(\mathcal{C})$ follows often from a grading on the category $\mathcal{C}$ such that $\varepsilon$ sits in non-trivial degree.

\medskip

We can now easily give an answer to \cref{qu:change_signs} for the fusion algebra $A_{\tilde{\mathbb{S}}}\simeq \Gr(\mathcal{C})/([\varepsilon]+[\mathbf{1}])$. Indeed, the answer is positive if and only if the slightly degenerate category $\mathcal{C}$ is equivalent to $\mathcal{C}_0 \boxtimes \sVect$, where $\sVect$ is the category of super vector spaces and $\mathcal{C}_0$ is a non-degenerate braided category. From the categorical point of view, changing a sign of a basis element $[X]$ of $\Gr(\mathcal{C})/([\varepsilon]+[\mathbf{1}])$ amounts to pick $\varepsilon \otimes X$ instead of $X$ in the set $J$.

\medskip

To a slightly degenerate category $\mathcal{C}$ as above, one can attach a non-degenerate braided pivotal supercategory $\tilde{\mathcal{C}}$ by adding an odd isomorphism between $X$ and $X\otimes \varepsilon$, see \cite[Section 4]{slightly-deg} for more details. This procedure can be thought as a super version of modularization procedure for degenerate braided pivotal fusion categories, due to Bruguières \cite{bruguieres} and independently Müger \cite{muger-galois}. In this case, the ring $\Gr(\mathcal{C})/([\varepsilon]+[\mathbf{1}])$ is seen as the super Grothendieck ring of $\tilde{\mathcal{C}}$.



\section{Quantum $\gl_n$ and its representations}
\label{sec:qgl}

In order to produce slightly degenerate categories, we consider categories of representations of the universal enveloping algebra of the reductive Lie algebra $\mathfrak{gl}_n$, with the deformation parameter being an even root of unity. The semisimplification of the category of tilting modules will provide a semisimple category, as in the case of a simple Lie algebra, but with an infinite number of simple objects. Nevertheless, the symmetric center has an infinite number of non-isomorphic simple objects, and killing the one-dimensional transparent objects will produce a non-degenerate of a slightly degenerate category. 

\subsection{Root system for $\gl_n$}
\label{sec:root_system}

We set up some notations for the root system of $\gl_n$. Let $P=\mathbb{Z}^n$ be the weight lattice of $\gl_n$ with standard basis $\varepsilon_1,\ldots,\varepsilon_n$. We equip it with the usual scalar product, which is given by $\langle\varepsilon_i,\varepsilon_j\rangle = \delta_{i,j}$. We also define the simple roots $\alpha_i = \varepsilon_i-\varepsilon_{i+1}$ for every $1\leq i < n$, which span over $\mathbb{Z}$ the root lattice $Q$. Fro $1 \leq j \leq n$, let $\varpi_j=\varepsilon_1+\cdots+\varepsilon_j$ be the fundamental roots which satisfy for every $1 \leq i < n$ and $1 \leq j \leq n$, $\langle\varpi_j,\alpha_i\rangle=\delta_{i,j}$.

The symmetric groups in $n$ letters $W$ acts on $P$ by permuting the coordinates and the scalar product $\langle -,- \rangle$ is $W$-equivariant. Let $l$ be the length function of $W$ for its usual Coxeter structure and we denote by $w_0$ the longest element of $W$ which is given by $w_0(k) = n+1-k$.  

Let $P^+$ be the set of dominant integral weights
\[
  P^+=\left\{\sum_{i=1}^n\lambda_i\varepsilon_i\in P\ \middle\vert\ \lambda_1\geq \lambda_2\geq\cdots\geq \lambda_n\right\}.
\]
Note that every fundamental weight is in $P^+$. Let $\rho=\varpi_1+\ldots+\varpi_{n-1}=\sum_{i=1}^n(n-i)\varepsilon_i$, which is again an element of $P^+$.

\subsection{Rational form of quantum $\gl_n$}
\label{sec:rational_form}

In this section, we fix $q$ an indeterminate over $\mathbb{Z}$, let $\mathcal{A}=\mathbb{Z}[q,q^{-1}]$ and $\Bbbk=\mathbb{Q}(q)$ its field of fractions. In $\mathcal{A}$, we define the following elements
\begin{align*}
  [n] &= \frac{q^n-q^{-n}}{q-q^{-1}}, & [n]! &= \prod_{i=1}^n[i], & \qbinom{n}{k} &= \frac{[n]!}{[k]![n-k]!},
\end{align*}
for any $n \in \mathbb{N}$ and $0 \leq k \leq n$.

\begin{definition}
  The \emph{quantum enveloping algebra} $\mathcal{U}_q(\gl_n)$ of $\gl_n$ is the $\mathbb{Q}(q)$ algebra generated by $E_i,F_i,L_j^{\pm 1}$ for $1\leq i < n$ and $1\leq j \leq n$ subject to the following relations
  \begin{align*}
    L_iL_i^{-1}&=L_i^{-1}L_i=1, & L_iL_j&=L_jL_i,\\
    L_iE_j &= q^{\langle\varepsilon_i,\alpha_j\rangle}E_jL_i, & L_iF_j &= q^{-\langle\varepsilon_i,\alpha_j\rangle}F_jL_i,
  \end{align*}
  \[
        [E_i,F_j] = \delta_{i,j}\frac{K_i-K_i^{-1}}{q-q^{-1}},
  \]
  where $K_i=L_iL_{i+1}^{-1}$, and subject to the quantum Serre relations
  \begin{align*}
    E_iE_j&=E_jE_i, & F_iF_j&=F_jF_i, & \text{if }\lvert i-j\rvert > 1, \\
    E_i^2E_j&-[2]E_iE_jE_i + E_jE_i^2 = 0, & F_i^2F_j&-[2]F_iF_jF_i + F_jF_i^2 = 0, & \text{if }\lvert i-j\rvert = 1.
  \end{align*}
\end{definition}

Note that $\mathcal{U}_q(\mathfrak{sl}_n)$ is isomorphic to the $\mathbb{Q}(q)$-subalgebra generated by $E_i,F_i$ and $K_i$ for $1\leq i < n$. For any $\lambda\in P$, we also define $L_\lambda = \prod_{i=1}^nL_i^{\lambda_i}$ so that $L_i=L_{\varepsilon_i}$ and $K_i=L_{\alpha_i}$. It is trivial, but nonetheless crucial, to check that $L_{\varpi_n}$ is a central element in $\mathcal{U}_q(\gl_n)$.

Let $\mathcal{U}_q(\gl_n)^{<0}$ (resp. $\mathcal{U}_q(\gl_n)^{\leq 0}$)  be the subalgebra generated by $(F_i)_{1 \leq i < n}$ (resp. $(F_i,L_j)_{1\leq i < n, 1 \leq j \leq n}$), $\mathcal{U}_q(\gl_n)^{>0}$ (resp. $\mathcal{U}_q(\gl_n)^{\geq 0}$) be the subalgebra generated by $(E_i)_{1 \leq i < n}$ (resp. $(E_i,L_j)_{1\leq i < n, 1 \leq j \leq n}$) and $\mathcal{U}_q(\gl_n)^0$ be the subalgebra generated by $(L_i)_{1 \leq i \leq n}$. It is a well-known fact that $\mathcal{U}_q(\gl_n)$ has the following triangular decomposition, as a $\mathbb{Q}(q)$-vector space
\[
  \mathcal{U}_q(\gl_n) \simeq \mathcal{U}_q(\gl_n)^{<0} \otimes \mathcal{U}_q(\gl_n)^0 \otimes \mathcal{U}_q(\gl_n)^{>0},
\]
the isomorphism being given by multiplication.

We endow the algebra $\mathcal{U}_q(\gl_n)$ with a comultiplication $\Delta$, a counit $\varepsilon$ and an antipode $S$ which turn $\mathcal{U}_q(\gl_n)$ into a Hopf algebra. These are given on the generators by
\begin{align*}
  \Delta(E_i) &= E_i\otimes K_i+1\otimes E_i,& \Delta(F_i) &= F_i\otimes 1 + K_i^{-1}\otimes F_i,& \Delta(L_i)&=L_i\otimes L_i,\\
  \varepsilon(E_i) &=0, & \varepsilon(F_i) &= 0, & \varepsilon(L_i) &= 1,\\
  S(E_i)&= -E_iK_i^{-1},& S(F_i) &= -K_iF_i,& S(L_i)&=L_i^{-1}.
\end{align*}

We also note that $S^2$ is given by conjugation by $L_{2\rho}$: for any $x\in\mathcal{U}_q(\gl_n)$,
\[
  S^{2}(x) = L_{2\rho}xL_{2\rho}^{-1}.
\]

The quantum group $\mathcal{U}_q(\gl_n)$ has the same quasi-$R$-matrix as $\mathcal{U}_q(\mathfrak{sl}_n)$. It is an element $\Theta=\sum_{\lambda\in Q\cap P^+}\Theta_\lambda$ in a completion of $\mathcal{U}_q(\gl_n)^{>0}\otimes \mathcal{U}_q(\gl_n)^{<0}$, see \cite[Chapter 7]{jantzen} for more details. We just give here some important properties of this quasi-$R$-matrix. Let $\Psi$ be the algebra automorphism of $\mathcal{U}_q(\gl_n)\otimes \mathcal{U}_q(\gl_n)$ given by
\begin{align*}
  \Psi(E_i\otimes 1) &= E_i\otimes K_i^{-1},& \Psi(F_i\otimes 1) &= F_i\otimes K_i,& \Psi(L_i\otimes 1) &= L_i\otimes 1,\\
  \Psi(1\otimes E_i) &= K_i^{-1}\otimes E_i,& \Psi(1\otimes F_i) &= K_i\otimes F_i,& \Psi(1\otimes L_i) &= 1\otimes L_i.\\
\end{align*}
Then one has
\begin{equation}
  \label{eq:theta_comult}
  \Theta \Delta(x) = (\Psi\circ \Delta^{\op})(x)\Theta,
\end{equation}
for any $x\in \mathcal{U}_q(\gl_n)$, where $\Delta^{\op}$ denotes the opposite comultiplication. Moreover $\Theta$ is invertible and satisfies
\begin{equation}
  \label{eq:theta_yb}
  (\Delta\otimes \id)(\Theta) = \Psi_{23}(\Theta_{13})\Theta_{23}\quad\text{and}\quad(\id\otimes \Delta)(\Theta) = \Theta_{12}(\Theta_{13})\Theta_{12}.
\end{equation}

Finally, one may give an explicit form of $\Theta$, see for example \cite[\S 10.1.D]{chari-pressley}.

\subsection{Lusztig's restricted integral form}
\label{sec:integral_form}

Following \cite[\S 9.3.A]{chari-pressley}, we define an integral version of $\mathcal{U}_q(\gl_n)$ over $\mathcal{A}$, which will be suitable for specializations at roots of unity.

\begin{definition}
  The Lusztig's restricted integral form $\mathcal{U}_q^{\mathcal{A}}(\gl_n)$ of $\mathcal{U}_q(\gl_n)$ is the $\mathcal{A}$-subalgebra of $\mathcal{U}_q(\gl_n)$ generated by
  \begin{align*}
    E_i^{(n)} &= \frac{E^n}{[n]!}, & F_i^{(n)}&=\frac{F^n}{[n]!}, &L_j& & \text{and}&& \qbinom{L_j ; c}{t} &= \prod_{s=1}^{t}\frac{q^{c+1-s}L_j-q^{s-c-1}L_j^{-1}}{q^s-q^{-s}},
  \end{align*}
  for $1\leq i < n$ and $1 \leq j \leq n$.
\end{definition}

We denote by $\mathcal{U}_q^{\mathcal{A}}(\gl_n)^{?} = \mathcal{U}_q(\gl_n)^{?}\cap \mathcal{U}_q^{\mathcal{A}}(\gl_n)$ for $? \in \{{<}0,{\leq}0 , 0 ,{\geq}0, {>}0\}$. The restricted integral form still has a triangular decomposition as an $\mathcal{A}$-module
\[
  \mathcal{U}_q^{\mathcal{A}}(\gl_n) \simeq \mathcal{U}_q^{\mathcal{A}}(\gl_n)^{<0} \otimes \mathcal{U}_q^{\mathcal{A}}(\gl_n)^0 \otimes \mathcal{U}_q^{\mathcal{A}}(\gl_n)^{>0}.
\]

The comultiplication, counit and antipode restricts to the integral form and endow it with a structure of a Hopf algebra. Using the explicit form of the quasi-$R$-matrix $\Theta$, one may show that it lies in (a completion of) $\mathcal{U}_q^{\mathcal{A}}(\gl_n)^{>0}\otimes \mathcal{U}_q^{\mathcal{A}}(\gl_n)^{<0}$. 

\subsection{Representations}
\label{sec:rep}

Since $\mathcal{U}_q(\gl_n)$ is a Hopf algebra, the tensor product of two $\mathcal{U}_q(\gl_n)$-modules is still an $\mathcal{U}_q(\gl_n)$-module. Using the antipode $S$, we also equip the dual $V^*$ of an $\mathcal{U}_q(\gl_n)$-module $V$ with a structure of an $\mathcal{U}_q(\gl_n)$-module:
\[
  (x\cdot\varphi)(v) = \varphi(S(x)\cdot v), 
\]
for any $x\in\mathcal{U}_q(\gl_n)$, $v\in V$ and $\varphi \in V^*$.

\subsubsection{Rational representations}

For an $\mathcal{U}_q(\gl_n)$-module $M$ and $\lambda\in P$, we define the $\lambda$-weight space of $M$ as
\[
  M_\lambda=\left\{m\in M\ \middle\vert\ L_im = q^{\langle\lambda,\varepsilon_i\rangle},\ \text{for all } 1 \leq i \leq n\right\}.
\]

We will only consider weight modules of type $1$: these modules are direct sums of their weight spaces as defined above. For any $\lambda\in P$, the Verma module of highest weight $\lambda$ is
\[
  M(\lambda) = \mathcal{U}_q(\gl_n)\otimes_{\mathcal{U}_q(\gl_n)^{\geq 0}} \mathbb{Q}(q)v_\lambda,
\]
where $\mathbb{Q}(q)v_\lambda$ is the one dimensional representation of $\mathcal{U}_q(\gl_n)^{\geq 0}$ given by
\[
  E_i\cdot v_\lambda = 0 \quad \text{and} \quad L_j\cdot v_\lambda = q^{\langle\varepsilon_j,\lambda\rangle}v_\lambda,
\]
for every $1\leq i < n$ and $1 \leq j \leq n$.

\begin{proposition}
  If $\lambda\in P^+$ then $M(\lambda)$ has a unique irreducible finite dimensional quotient $L(\lambda)$. Moreover, every irreducible finite dimensional weight module is isomorphic to a $L(\lambda)$ for a unique $\lambda\in P^+$: irreducible finite dimensional weight modules are parameterized by $P^+$.
\end{proposition}

For $\lambda=k\varpi_n$, it is easy to see that $L(\lambda)$ is one-dimensional, and we will denote by $\moddet{q}$ the module $L(\varpi_n)$. Note that $\moddet{q}^*\simeq L(-\varpi_n)$ and therefore setting $\moddet{q}^{\otimes k} = L(k\varpi_n)$ is coherent with the fact that $\moddet{q}^k\otimes\moddet{q}^l\simeq\moddet{q}^{k+l}$.

It is clear that if $\lambda\in P^+$ then $-w_0(\lambda)\in P^+$. Moreover, we have $L(\lambda)^*\simeq L(-w_0(\lambda))$.

\subsubsection{A $\mathbb{Z}$-grading}
\label{sec:grading}

Since the element $L_{\varpi_n}$ is central in $\mathcal{U}_q(\mathfrak{gl}_n)$, it induces a $\mathbb{Z}$-grading on the category of finite dimensional representations of $\mathcal{U}_q(\mathfrak{gl}_n)$: a simple object $L(\lambda)$ is of degree $\langle\lambda,\varpi_n\rangle$. Since $L_{\varpi_n}$ is group-like, this grading is of course compatible with the tensor product: every simple summand of a $L(\lambda)\otimes L(\mu)$ is of degree $\langle\lambda,\varpi_n\rangle + \langle\mu,\varpi_n\rangle$. 

\subsubsection{Braiding and pivotal structures}
\label{sec:braiding}

Using the quasi-$R$-matrix $\Theta$, we define a braiding on the category of finite dimensional weight modules. For $M$ and $M'$ two finite dimensional weight modules, we let $\Theta_{M,M'}\colon M\otimes M' \rightarrow M\otimes M'$ be the $\mathbb{Q}(q)$-linear isomorphism given by the action of $\Theta$. Since $M$ and $M'$ are finite dimensional, only a finite number of $\Theta_\lambda$ acts non-trivially: $\Theta_\lambda(M_\mu\otimes M'_{\mu'})\subset M_{\mu+\lambda}\otimes M'_{\mu'-\lambda}$.

Now consider the map $f_{M,M'}\colon M\otimes M' \rightarrow M\otimes M'$ given by
\[
  f_{M,M'}(m\otimes m') = q^{\langle\mu,\mu'\rangle}m\otimes m'
\]
for every $m\in M_\mu$ and $m\in M'_{\mu'}$. One easily check on the generators of $\mathcal{U}_q(\gl_n)$ that for every $x\in \mathcal{U}_q(\gl_n)\otimes \mathcal{U}_q(\gl_n)$, any $m\in M$ and $m\in M'$ one has
\begin{equation}
  \label{eq:action_f}
  x\cdot f(m\otimes m') = f(\Psi(x)\cdot (m\otimes m')).
\end{equation}

We then define $c_{M,M'} = \tau\circ f_{M,M'} \circ \Theta_{M,M'}$, where $\tau(m\otimes m') = m'\otimes m$. Combining \eqref{eq:theta_comult} and \eqref{eq:action_f}, one obtains that $c_{M,M'}$ is an $\mathcal{U}_q(\gl_n)$-equivariant map.

Using \eqref{eq:theta_yb} and \eqref{eq:action_f}, one shows that $c_{-,-}$ satisfy the hexagon axioms:
\[
  c_{L,M\otimes N}=(\id_M\otimes c_{L,N})\circ(c_{L,M}\otimes \id_{N})\quad\text{and}\quad c_{L\otimes M,N}=(c_{L,N}\otimes \id_M)\circ(\id_L\otimes c_{M,N}).
\]

\begin{proposition}
  The family of maps $c_{M,M'}$ endow the category of finite dimensional weight modules over $\mathcal{U}_q(\gl_n)$ with a structure of a braided category.
\end{proposition}

\medskip

We now turn to pivotal structures on the category of finite dimensional weight modules over $\mathcal{U}_q(\gl_n)$. Since the square of the antipode is given by conjugation by $L_{2\rho}$, the map $a_M \colon M \rightarrow M^{**}$ sending $m\in M$ to $a_M(m)\in M^{**}$ defined by
\[
  a_M(m)(\varphi) = \varphi(L_{2\rho} \cdot m)
\]
for any $\varphi \in M^*$ is an isomorphism of $\mathcal{U}_q(\gl_n)$-modules. But conjugation by the element $L_{2\rho}L_{k\varpi_n}$ also gives the square of the antipode since $L_{k\varpi_n}$ is central. We define therefore analogously an isomorphism $a_{k,M} \colon M \rightarrow M^{**}$.

\begin{proposition}
  The family of maps $a_{k,M}$ endow the category of finite dimensional weights modules over $\mathcal{U}_q(\gl_n)$ with a structure of a pivotal category. Moreover, the pivotal structure is spherical if and only if $k=1-n$.
\end{proposition}

\begin{proof}
  The twist associated to the pivotal structure $a_k$ is given on the simple object $L(\lambda)$ by multiplication by $q^{\langle\lambda,\lambda+2\rho+p\varpi_n\rangle}$. Since $L(\lambda)^*\simeq L(-w_0(\lambda))$, the pivotal structure is spherical if and only if for all $\lambda\in P^+$, we have $q^{\langle\lambda,\lambda+2\rho+k\varpi_n\rangle}=q^{\langle-w_0(\lambda),-w_0(\lambda)+2\rho+k\varpi_n\rangle}$. As $q$ is an indeterminate over $\mathbb{Q}$ the last condition is equivalent to $\langle\lambda,\lambda+2\rho+k\varpi_n\rangle=\langle-w_0(\lambda),-w_0(\lambda)+2\rho+k\varpi_n\rangle$. But $-w_0(\rho) = \rho+(1-n)\varpi_n$ so that
  \[
    \langle-w_0(\lambda),-w_0(\lambda)+2\rho+k\varpi_n\rangle = \langle\lambda,\lambda+2\rho+k\varpi_n\rangle +2((1-n)-k)\langle\lambda,\varpi_n\rangle.
  \]
  Therefore the pivotal structure $a_{k,-}$ is spherical if and only if $k=1-n$. 
\end{proof}

The right quantum trace with respect to the pivotal structure $a_{k,-}$ is denoted $\Tr_k$, the associated quantum dimension by $\dim_k$ and the associated twist by $\theta_{k,-}$. 

\subsubsection{Integral representations}

Since we work with Lusztig's restricted integral form, we need to adapt slightly the definition of a weight space. Given $M$ an $\mathcal{U}_q^{\mathcal{A}}(\gl_n)$-module and $\lambda\in P$, the $\lambda$-weight space of $M$ is
\[
  M_\lambda=\left\{m\in M\ \middle\vert\ L_i\cdot m = q^{\langle\lambda,\varepsilon_i\rangle}m,\ \qbinom{L_i;0}{t}\cdot m = \qbinom{\langle\lambda,\varepsilon_i\rangle}{t} m,\text{ for all } 1 \leq i \leq n \text{ and } t\in \mathbb{N}\right\}.
\]
We will again only consider weight modules, that is modules which are sum of their weight spaces.

There also exists an integral version of the representations $M(\lambda)$ and $L(\lambda)$. Denote by $M^{\mathcal{A}}(\lambda)$ (resp. $L^{\mathcal{A}}(\lambda)$) the $\mathcal{U}_q^{\mathcal{A}}(\gl_n)$-submodule of $M(\lambda)$ (resp. of $L(\lambda)$) generated by $v_\lambda$. Then
\[
  M^{\mathcal{A}}(\lambda)\otimes_{\mathcal{A}}\mathbb{Q}(q) \simeq M(\lambda)\quad\text{and}\quad L^{\mathcal{A}}(\lambda)\otimes_{\mathcal{A}}\mathbb{Q}(q) \simeq L(\lambda).
\]
The module $L(\lambda)$ is the integral Weyl module of highest weight $\lambda$.

\medskip

Similarly to the braiding described above for finite dimensional weight modules over $\mathcal{U}_q(\gl_n)$ one defines a braiding structure on the category of finite dimensional weight modules over $\mathcal{U}_q^{\mathcal{A}}(\gl_n)$: we have already seen that the quasi-$R$-matrix $\Theta$ lies in the Lusztig's restricted integral form of $\mathcal{U}_q(\gl_n)$. One also has the family of pivotal structures $a_{k,-}$.

\subsection{Specialization, tilting modules and semisimplification}
\label{sec:tilting}

Let $d> 0$ be an integer and $\xi=\exp(i\pi/d)$. We define the quantum group $\mathcal{U}_\xi(\gl_n)$ as the specialization of Lusztig's restricted integral form:
\[
  \mathcal{U}_\xi(\gl_n) = \mathcal{U}_q^{\mathcal{A}}(\gl_n)\otimes_{\mathcal{A}}\mathbb{C},
\]
where we see $\mathbb{C}$ as an $\mathcal{A}$-algebra via the $\mathbb{Z}$-linear map $q\mapsto \xi$. Since $\xi$ is a primitive $2d$-th root of unity, we have extra relations in this specialization: for example $E_i^d=F_i^d=0$.

\medskip

The construction of a fusion category from the quantum enveloping algebra of a simple Lie algebra extends to our situation with $\gl_n$, which is only a reductive Lie algebra. We recall quickly the main steps of this construction:
\begin{enumerate}
\item We have at our disposal the specialization of the Weyl module $L(\lambda)\otimes_{\mathcal{A}}\mathbb{C}$. We say that a module $M$ over $\mathcal{U}_\xi(\gl_n)$ is tilting if both $M$ and $M^*$ have a filtration by specializations of Weyl modules.
\item One shows that the category of tilting modules is stable under direct sum, tensor product and duality.
\item We semisimplify the monoidal category of tilting modules by killing negligible modules with respect to the pivotal structure $a_{0,-}$ (or equivalently any $a_{k,-}$), see \cite{etingof-ss} for a description of this procedure.
\end{enumerate}

When doing these steps with the quantum enveloping algebra of a simple Lie algebra, one obtains a fusion category, with a braiding and a pivotal (even spherical) structure. Here, since we work with $\gl_n$, we do not have a finite number of simple objects. Let us denote by $\mathcal{C}_\xi$ the category obtained with this procedure. One shows that the simple objects are given by the images of the $L(\lambda)\otimes_{\mathcal{A}} \mathbb{C}$ for $\lambda$ in the alcove
\[
  C_{n,d} =\left\{ \lambda\in P^+\ \middle\vert \langle\lambda,\varepsilon_1-\varepsilon_n\rangle \leq d-n\right\}. 
\]
In order to distinguish the simple objects of $\mathcal{C}_\xi$ with the Weyl modules, we denote by $X(\lambda)$ the simple object in $\mathcal{C}_\xi$ parameterized by $\lambda\in C_{n,d}$. We nonetheless use the notation $\moddet{\xi}^{\otimes k}$ for $X(k\varpi_n)$. The category $\mathcal{C}_\xi$ inherits the $\mathbb{Z}$-grading from \cref{sec:grading}. Note that this grading is not obtained from the action of the element $L_{\varpi_n}$ since $L_{\varpi_n}^{2d} = 1$ and we would only obtain a grading by the group $\mathbb{Z}/2d\mathbb{Z}$.

Note that $\mathcal{C}_\xi$ is non-zero if and only if $d\geq n$. From now on, we will always work under this assumption. Even if $d=n$, we have an infinite number of simple objects since $\varpi_n$ and its multiples are in $C_{n,d}$. The tensor product of $X(\varpi_i)$ with $X(\lambda)$ is given by
\begin{equation}
  \label{eq:tensor_fundamental}
  X(\varpi_i)\otimes X(\lambda) \simeq \bigoplus_{\substack{1\leq j_1 < \cdots < j_i \leq n\\\lambda+\varepsilon_{j_1}+\cdots+\varepsilon_{j_i}\in C_{n,d}}}X(\lambda+\varepsilon_{j_1}+\cdots+\varepsilon_{j_i}),
\end{equation}
see \cite[(3.2)]{andersen-stroppel}. It is also easy to check that the invertible objects are exactly of the form $X((d-n)\varpi_i+r\varpi_n)$ for some $1 \leq i \leq n$ and $r\in \mathbb{Z}$. Moreover,
\[
  X((d-n)\varpi_i+r\varpi_n) \otimes X(\lambda) \simeq X(\sh^i(\lambda) + (d-n)\varpi_i+r\varpi_n),
\]
where $\sh(\lambda) = \lambda_n\varepsilon_1 + \sum_{i=1}^{n-1}\lambda_i\varepsilon_{i+1}$ if $\lambda=\sum_{i=1}^n\lambda_i\varepsilon_i$.

The category $\mathcal{C}_\xi$ is braided, with braiding given by (the image of) $c_{-,-}$. We also have at our disposal several pivotal structures, given by (the image of) $a_{k,-}$. As before, the right quantum trace with respect to the pivotal structure $a_{k,-}$ is denoted $\Tr_k$ and the quantum dimension by $\dim_k$. 

\begin{proposition}
  \label{prop:S-mat-gln}
  The $S$-matrix of $\mathcal{C}_\xi$ is given by
  \[
    S_{X(\lambda),X(\mu)}=\xi^{\langle\lambda+\mu,k\varpi_n\rangle}\frac{\sum_{w\in W}(-1)^{l(w)}\xi^{2\langle w(\lambda+\rho),\mu+\rho\rangle}}{\sum_{w\in W}(-1)^{l(w)}\xi^{2\langle w(\rho),\rho\rangle}},
  \]
  for all $\lambda,\mu\in C_{n,d}$.

  The value of the twist associated with the pivotal structure $a_{k,-}$ on the simple object $X(\lambda)$ is given by $\xi^{\langle\lambda,\lambda+2\rho+k\varpi_n\rangle}$.  
\end{proposition}

\begin{proof}
  It suffices to do it for the pivotal structure $a_{0,-}$. Indeed, the quantum trace $\Tr_k$ of an element $f\in\End(X)$ is given by
  \[
    \Tr_k(f) = \Tr(L_{k\varpi_n}L_{2\rho}f),
  \]
  where $\Tr$ is the usual trace. Since $L_{k\varpi_n}$ is group-like and central, it acts by a scalar on $X(\lambda)\otimes X(\mu)$, and it is easy to check on the highest weight vector $v_\lambda\otimes v_\mu$ that $L_{k\varpi_n}\cdot v_\lambda\otimes v_\mu= \xi^{\langle\lambda+\mu,k\varpi_n\rangle}v_\lambda\otimes v_\mu$. The formula for $p=0$ is obtained using the same arguments as in \cite[Theorem 3.3.20]{bakalov-kirillov}.
\end{proof}

Note that since $\xi$ is a $2d$-th root of unity, the pivotal structure $a_{k,-}$ is spherical if $k\equiv 1-n\ [d]$ and $a_{k,-}=a_{l,-}$ if $k\equiv l\ [2d]$ since $L_{2d\varpi_n}$ acts by $1$ on any $X(\lambda)$.

\subsection{Symmetric center and modularization}
\label{sec:symmetric_center}

The category $\mathcal{C}_\xi$ has an infinite number of simple objects and we aim to produce out of it a fusion or superfusion category with invertible $S$-matrix. Therefore we need to determine the symmetric center of $\mathcal{C}_\xi$ because having an invertible $S$-matrix is equivalent to having a trivial symmetric center.

\begin{proposition}
  \label{prop:symmetric_center}
  The simple objects belonging to symmetric center of $\mathcal{C}_\xi$ are the $X((d-n)\varpi_i+ r\varpi_n)$ with $1 \leq i \leq n$ and $r\equiv i \ [d]$. The symmetric center of $\mathcal{C}_\xi$ is then pointed and tensor generated by $X((d-n)\varpi_1+\varpi_n)$.
\end{proposition}

\begin{proof}
  We will use the same strategy as in \cite[Section 4]{bruguieres} by proving first that a transparent simple object is invertible and then detecting the transparent objects among the invertible ones.

  \medskip
  
  Let $X$ be a transparent simple object. Then it satisfies $c_{Y,X}\circ c_{X,Y}=\id_{X\otimes Y}$ for all object $Y$. We now use the ribbon $\theta_{1-n,-}$ and obtain that $\theta_{1-n,X}\otimes \theta_{1-n,Y} = \theta_{1-n,X\otimes Y}$. Now suppose that moreover $Y$ is simple, and that $X\otimes Y \simeq \bigoplus_{i}Z_i$. Taking the quantum trace $\Tr_{1-n}$, we obtain that $\theta_{1-n,X}\theta_{1-n,Y}\sum_{i}\dim_{1-n}(Z_i)=\sum_{i}\dim_{1-n}(Z_i)$. The twist being a power of $\xi$, by taking the norm we find that
  \[
    \left\lvert\sum_{i}\dim_{1-n}(Z_i)\right\rvert = \left\lvert\sum_{i}\theta_{1-n,Z}\dim_{1-n}(Z_i)\right\rvert.
  \]
  But with the spherical structure $a_{1-n,-}$, the quantum dimension of any simple object is a positive real number, and therefore
  \[
    \sum_{i}\dim_{1-n}(Z_i) = \left\lvert\sum_{i}\theta_{1-n,Z_i}\dim_{1-n}(Z_i)\right\rvert.
  \]
  By the case of equality in the triangle inequality, we deduce that the argument of $\theta_{1-n,Z_i}\dim_{1-n}(Z_i)$ does not depend on $i$. Since $\dim_{1-n}(Z_i)>0$ and $\theta_{1-n,Z_i}$ is a root of unity, we deduce that $\theta_{1-n,Z_i}$ does not depend on $i$ so that $\theta_{1-n,X\otimes Y}$ is a scalar multiple of $\id_{X\otimes Y}$. It is then also true for any twist $\theta_{k,-}$: we check that $\theta_{k,X(\lambda)} = \xi^{(k+n-1)\langle\lambda,\varpi_n\rangle}\theta_{1-n,X(\lambda)}$ and it remains to check that $\langle\lambda,\varpi_n\rangle$ only depends on $X$ and $Y$ if $X(\lambda)$ is a summand of $X\otimes Y$. But this is immediate since $\langle\lambda,\varpi_n\rangle$ is equal to the degree of $X\otimes Y$ for the $\mathbb{Z}$-grading.

  Now, we show that a simple object $X$ such that $\theta_{0,X\otimes Y}$ is a scalar for any simple object $Y$ is invertible. If $d=n$ then any simple object is invertible, so we may and will suppose that $d> n$. Suppose that $X\simeq X(\lambda)$, we take $Y=X(\varpi_1)$ and look at the decomposition of $X\otimes X(\varpi_1)$. Thanks to \eqref{eq:tensor_fundamental}, it is given by
  \[
    X(\lambda)\otimes X(\varpi_1) \simeq \bigoplus_{\substack{1 \leq j \leq n\\\lambda+\varepsilon_j\in C_{n,d}}}X(\lambda+\varepsilon_j).
  \]
  Suppose that there exists $1\leq i < j \leq n$ such that $\lambda+\varepsilon_i\in C_{n,d}$ and $\lambda+\varepsilon_j \in C_{n,d}$. Since $\theta_{0,X\otimes Y}$ is a scalar, we have $\theta_{0,X(\lambda+\varepsilon_i)}=\theta_{0,X(\lambda+\varepsilon_j)}$, that is
  \[
    \langle\lambda+\varepsilon_i,\lambda+\varepsilon_i+2\rho\rangle \equiv \langle\lambda+\varepsilon_j,\lambda+\varepsilon_j+2\rho\rangle\ [2d],
  \]
  which is equivalent to $2\lambda_i+2(n-i)+1\equiv 2\lambda_j+2(n-j)+1\ [2d]$. Then $\lambda_{i}-\lambda_{j} \equiv i-j\ [d]$ and since $d-n \geq \lambda_i-\lambda_j \geq 0$ we have $\lambda_i-\lambda_j = d+i-j$. As $0 > i-j > -n$, this leads to a contradiction.

  Hence there exists a unique $1 \leq i \leq n$ such that $\lambda+\varepsilon_i \in C_{n,d}$ and this is possible if and only if $\lambda=(d-n)\varpi_i+r\varpi_n$ for some $r$ and $X$ is thus invertible.

  \medskip

  We finally determine the transparent objects among the invertible ones. Let $1 \leq i \leq n$ and $X=X((d-n)\varpi_i+r\varpi_n)$ be an invertible object. Since $X\otimes X(\lambda) \simeq X(\sh^i(\lambda) + (d-n)\varpi_i + r\varpi_n)$, the object $X$ is transparent if and only if $\theta_{0,X(\sh^i(\lambda) + (d-n)\varpi_i + r\varpi_n)} = \theta_{0,X}\theta_{0,X(\lambda)}$ for all $\lambda\in C_{n,d}$. This last equality is equivalent to
  \begin{multline*}
    \langle\sh^i(\lambda) + (d-n)\varpi_i + r\varpi_n,\sh^i(\lambda) + (d-n)\varpi_i + r\varpi_n+2\rho\rangle \equiv \\
    \langle\lambda,\lambda+2\rho\rangle+\langle(d-n)\varpi_i + r\varpi_n,(d-n)\varpi_i + r\varpi_n+2\rho\rangle\ [2d].
  \end{multline*}
  Going back to the definition of $\rho$, we find that $X$ is transparent if and only of for all $\lambda\in C_{n,d}$ we have $(r-i)\langle\lambda,\varpi_n\rangle\equiv 0 \ [d]$. If $n=d$ there is no condition on $r$ and $i$ and then every object is transparent. If $d>d$, taking $\lambda=\varpi_1$, we see that $r\equiv i\ [d]$ so that the only transparent objects in $\mathcal{C}_\xi$ are the $X((d-n)\varpi_i+r\varpi_n)$ with $r\equiv i \ [d]$.
  Finally, we need to check that $X((d-n)\varpi_i+r\varpi_n)$ is a tensor power of $X((d-n)\varpi_1+\varpi_n)$. We remark that $X((d-n)\varpi_1+\varpi_n)^{\otimes n}\simeq \det_q^{\otimes d}$, so that, by tensoring to a suitable power of $\det_q^{\otimes d}$, we may and will suppose that $0 \leq r < d$. Then $i=r$ and $X((d-n)\varpi_i+i\varpi_n)\simeq X((d-n)\varpi_1+\varpi_n)^{\otimes i}$.
\end{proof}

If we want to kill the symmetric center by a process of modularization, we have to check that every simple object in the symmetric center is of twist $1$ for the chosen pivotal structure. Let us denote by $\varepsilon$ the object $X((d-n)\varpi_1 + \varpi_n)$ which tensor generates the symmetric center.

\begin{lemma}
  \label{lem:gen_center}
  We endow $\mathcal{C}_\xi$ with the pivotal structure $a_{k,-}$. The quantum dimension and the twist of $\varepsilon$ are respectively $(-1)^{k+n-1}$ and $(-1)^{d+k}$.
\end{lemma}

\begin{proof}
  Since $\varepsilon$ is an invertible object, the quantum dimension of $\varepsilon$ with respect to the spherical structure $a_{1-n,-}$ is $1$ since the quantum dimension of a simple object is positive with respect to this spherical structure. Therefore the quantum dimension of $\varepsilon$ with respect to the pivotal structure $a_{k,-}$ is $\xi^{\langle(d-n)\varpi_1+\varpi_n,(k+n-1)\varpi_n\rangle}=(-1)^{k+n-1}$.

  Concerning the twist, we only have to check that
  \[
    \langle (d-n)\varpi_1+\varpi_n,(d-n)\varpi_1+\varpi_n+2\rho+k\varpi_n\rangle \equiv d(k+d)\ [2d],
  \]
  which is an easy computation.
\end{proof}

Since we want the twist of objects in the symmetric center to be equal to $1$, we will always choose from now on a pivotal structure of the form $a_{2p+d,-}$ with $p\in \mathbb{Z}$. The object $\varepsilon^{\otimes 2}$ is always of quantum dimension and twist equal to $1$ so that we can apply the modularization procedure of Bruguières and Müger. We then obtain a fusion category $\mathcal{D}_\xi$ with symmetric center generated by the image of $\varepsilon$. Tensoring by $\varepsilon^{\otimes 2}$ does not have fixed points on the set of simple objects since this object sits in non-trivial degree for the $\mathbb{Z}$-grading. Thus the modularization procedure only add isomorphism between some simple objects and the image in $\mathcal{D}_\xi$ of every simple object in $\mathcal{C}_{\xi}$ is still a simple object.

The structure of the symmetric center depends now on the parity of $n+d$:
\begin{itemize}
\item if $n\not\equiv d\ [2]$ then the symmetric center of $\mathcal{D}_\xi$ is equivalent, as a pivotal category to $\Rep(\mathbb{Z}/2\mathbb{Z})$ and $\mathcal{D}_\xi$ is then modularizable. We then obtain a non-degenerate fusion category $\tilde{\mathcal{D}}_\xi$ from the modularization of $\mathcal{D}_\xi$.
\item if $n\equiv d\ [2]$ then the symmetric center of $\mathcal{D}_\xi$ is equivalent, as a pivotal category to $\sVect$ and $\mathcal{D}_\xi$ is then slightly degenerate. We then obtain a non-degenerate superfusion category $\tilde{\mathcal{D}}_\xi$ from the supermodularization of $\mathcal{D}_\xi.$
\end{itemize}

The simple objects of $\tilde{D}_\xi$ are then parameterized by the orbits of $\Irr(\mathcal{C}_\xi)$ under tensorization by $\varepsilon$.

\subsection{Modular data arising from $\tilde{D}_\xi$}
\label{sec:mod_data_gln}

We turn to the computation of the modular invariants of the category $\tilde{\mathcal{D}}_\xi$. First, we determine a suitable subset of $C_{n,d}$ for the pasteurization of simple objects of $\tilde{\mathcal{C}_\xi}$. This amounts to choose a representative of each orbit of isomorphism classes of simple objects of $\mathbb{C}_\xi$ under tensorization by $\varepsilon$. We say that two weights $\lambda\sim_{\elem}\mu$ in $C_{n,d}$  if $\lambda=\sh(\mu)+(d-n)\varpi_1+\varpi_n$ or $\mu=\sh(\lambda)+(d-n)\varpi_1+\varpi_n$. We then define an equivalence relation $\sim$ on $C_{n,d}$ as the reflexive and transitive closure of $\sim_{\elem}$. It is almost immediate to see that if $\lambda\sim \mu$ then there exists $0 <q k \leq n$ and $r \in \mathbb{Z}$ such that $\mu=\sh^k(\lambda)+(d-n)\varpi_k +(rd+k)\varpi_n$. Note that $\lambda\sim\mu$ if and only if $X(\lambda)\simeq X(\mu)$ in $\tilde{\mathcal{D}}_\xi$.

\begin{lemma}
  \label{lem:equiv-weights}
  Let $\lambda\in C_{n,d}$. Then there exists a unique $\lambda'\in C_{n,d}$ such that $\lambda'\sim \lambda$ and $d-n\geq \lambda'_1\geq\cdots\geq \lambda'_n\geq 0$.
\end{lemma}

\begin{proof}
  First notice that for every $\lambda\in C_{n,d}$ we have $\lambda\sim \lambda+d\varpi_n$ since $\varepsilon^{\otimes n}\simeq \det_\xi^{\otimes d}$.

  Now fix $\lambda\in C_{n,d}$. By adding or subtracting a multiple of $d\varpi_n$, we may and will suppose that $d-n \geq \lambda_1 > -n$. If moreover $\lambda_n\geq 0$, we have nothing to prove. Otherwise, let $1\leq k \leq n$ be minimal such that $\lambda_k\leq k-n-1$. Since $\lambda_n < 0$ the integer $k$ is well defined, and since $\lambda_1>-n$, we have $k\geq 2$. Let $\lambda'=\sh^{n+1-k}(\lambda)+(d-n)\varpi_{n+1-k}+(n+1-k)\varpi_n$. Then $\lambda'\sim \lambda$ and we have $\lambda'_1 = \lambda_k + d+1-k \leq d-n$ be definition of $k$. Finally, $\lambda'_n = \lambda_{k-1}+n+1-k \geq 0$ again by definition of $k$.

  \medskip

  For the uniqueness, suppose that $\lambda\sim\mu$ with $0\leq \lambda_i \leq d-n$ and $0\leq \mu_i \leq d-n$ for all $1 \leq i \leq n$. If $d=n$, there is nothing to do since $\lambda=\mu=0$. Therefore, we suppose $d>n$. Since $\lambda\sim \mu$ we have $-(d-n) \leq \lambda_i-\mu_j \leq d-n$ for all $1 \leq i,j \leq n$. There also exists $0 < k \leq n$ and $r \in \mathbb{Z}$ such that $\mu=\sh^k(\lambda)+(d-n)\varpi_k + (rd+k)\varpi_n$. We want to show that $k=n$ and $r=-1$.
  
  Suppose that $k<n$. Therefore, by choosing suitably $i$ and $j$ we have $-(d-n) \leq rd+k \leq d-n$ and $-(d-n) \leq rd+k+d-n \leq d-n$. This implies that $-(d-n) \leq rd+k \leq 0$ and $0 \leq rd+k+d-n < d$. But as $0 \leq k+d-n < d$ we obtain that $r=0$ and then $k \leq 0$, which is a contradiction.

  Then $k=n$ and $\lambda = \mu + (r+1)d\varpi_n$. As $\lvert \lambda_i - \mu_i \rvert < d$, we necessarily have $r=-1$, which ends the prove of uniqueness.
\end{proof}

Let $\tilde{C}_{n,d}$ be the subset of $C_{n,d}$ consisting of dominant integral weights $\lambda$ satisfying $d-n\geq\lambda_1\geq\cdots\geq\lambda_n\geq 0$. We will therefore use the set $\tilde{C}_{n,d}$ to parameterize simple objects in $\tilde{\mathcal{D}}_\xi$. This also shows that the rank of $\tilde{\mathcal{D}}_\xi$ is equal to $\binom{d}{n}$.

\begin{lemma}
  \label{lem:cat-dimension}
  The categorical dimension of $\tilde{\mathcal{D}}_\xi$ is
  \[
    \dim(\tilde{\mathcal{D}}_\xi) = (-1)^{n(n-1)/2}\frac{d^n}{\left(\prod_{i=1}^{n-1}(\xi^i-\xi^{-i})^{n-i}\right)^2}.
  \]
\end{lemma}

\begin{proof}
  Since the simple object $\moddet{\xi}$ is invertible, tensoring by $\moddet{\xi}$ does not change the square norm $\lvert X \rvert^2$ of a simple object $X$.
  Note that we have a bijection between the simple objects of the fusion category associated with $\mathfrak{sl}_n$ and the set of $\lambda\in C_{n,d}$ with $\lambda_n=0$. Moreover, the squared norm of the object $X(\lambda)$ is the same if we restrict to $\mathfrak{sl}_n$. Therefore, by \cite[Theorem 3.3.20]{bakalov-kirillov}
  \[
    \sum_{\substack{\lambda\in C_{n,d}\\\lambda_n=0}}\lvert X(\lambda)\rvert ^2 = (-1)^{n(n-1)/2}\frac{nd^{n-1}}{\left(\prod_{i=1}^{n-1}(\xi^i-\xi^{-i})^{n-i}\right)^2}.
  \]
  Consider now the category $\mathcal{C}'_\xi$ obtained from $\mathcal{C}_\xi$ by adding isomorphisms between powers of $\moddet{\xi}^{2d}$. Its objects can be parameterized by $\{\lambda\in C_{n,d}\ \vert\ 0\leq \lambda_n < 2d\}$, so that its categorical dimension is equal to
  \[
    \dim(\mathcal{C}'_\xi)=2d \sum_{\substack{\lambda\in C_{n,d}\\\lambda_n=0}}\lvert X(\lambda)\rvert ^2 = (-1)^{n(n-1)/2}\frac{2nd^n}{\left(\prod_{i=1}^{n-1}(\xi^i-\xi^{-i})^{n-i}\right)^2}.
  \]
  Note that $\mathcal{C}_\xi'$ inherits of a $\mathbb{Z}/2dn\mathbb{Z}$-grading since $\det_\xi^{\otimes 2d}$ sits in degree $2dn$. As $\varepsilon$ is of order $2n$ in $\mathcal{C}'_\xi$ and none of its non-trivial tensor power is of degree $0$, we find that the (super)dimension of $\tilde{\mathcal{D}}_\xi$ is equal to $\dim(\mathcal{C}'_\xi)/2n$, which ends the proof.
\end{proof}

Since the category $\tilde{\mathcal{D}}_\xi$ is pivotal but not necessarily not spherical, we need to determine the invertible object $\bar{\mathbf{1}}$ among the $X(\lambda)$, $\lambda\in \tilde{C}_{n,d}$ and its quantum dimension.

\begin{proposition}
  \label{prop:1-bar}
  Let $p\in \mathbb{Z}$ and we equip $\tilde{\mathcal{D}}_\xi$ with the pivotal structure $a_{2p+d,-}$. The object $\bar{\mathbf{1}}$ is isomorphic to $\moddet{\xi}^{-(2p+n-1)}$.
\end{proposition}

\begin{proof}
  The object $\bar{\mathbf{1}}$ is determined by $\dim_{2p+d}(X(\lambda)^*)\dim_{2p+d}(\bar{\mathbf{1}}) = S_{\bar{\mathbf{1}},X(\lambda)}$ for all $\lambda\in C_{n,d}$. We first compute the quantum dimension of $X(\lambda)^*$ by using the fact that $X(\lambda)^*\simeq X(-w_0(\lambda))$
  \[
    \dim_{2p+d}(X(\lambda)^*) = \xi^{-\langle w_0(\lambda),(2p+d)\varpi_n\rangle}\frac{\sum_{w\in W}(-1)^{l(w)}\xi^{2\langle w(-w_0(\lambda)+\rho),\rho\rangle}}{\sum_{w\in W}(-1)^{l(w)}\xi^{2\langle w(\rho),\rho\rangle}}.
  \]

  Now we use that $w_0(\rho) = -\rho + (n-1)\varpi_n$  and the fact that $\varpi_n$ is invariant under the action of $W$ to obtain
  \[
    \langle w(-w_0(\lambda)+\rho),\rho\rangle = \langle ww_0(\lambda+\rho),w_0(\rho)\rangle - \langle \lambda,(n-1)\varpi_n\rangle.
  \]
  Therefore, by a change of variables in the sum at the numerator, we obtain
  \[
    \dim_{2p+d}(X(\lambda)^*) = \xi^{-\langle\lambda,(2p+d+2(n-1))\varpi_n\rangle}\frac{\sum_{w\in W}(-1)^{l(w)}\xi^{2\langle w(\lambda+\rho),\rho\rangle}}{\sum_{w\in W}(-1)^{l(w)}\xi^{2\langle w(\rho),\rho\rangle}}.
  \]
  
  It is now an easy calculation to show that
  \begin{align*}
    \frac{S_{\moddet{\xi}^{-(2p+n-1)},X(\lambda)}}{\dim_{2p+d}\left(\moddet{\xi}^{-(2p+n-1)}\right)}
    &= \xi^{-\langle\lambda,(2p+d+2(n-1))\varpi_n\rangle}\frac{\sum_{w\in W}(-1)^{l(w)}\xi^{2\langle w(\lambda+\rho),\rho\rangle}}{\sum_{w\in W}(-1)^{l(w)}\xi^{2\langle w(\rho),\rho\rangle}}\\
    &= \dim_{2p+d}(X(\lambda)^*),
  \end{align*}
  which ends the proof.
\end{proof}

The unique $\lambda\in \tilde{C}_{n,d}$ such that $-(2p+d+n-1)\varpi_n \sim \lambda$ is subtle to determine, but we only need the value of its quantum dimension, which is then $\pm\xi^{-2\langle (2p+n-1)\varpi_n,\rho+p\varpi_n\rangle}$.

\begin{theorem}
  \label{thm:modular_data_gln}
  We endow $\tilde{\mathcal{D}}_\xi$ with the pivotal structure $a_{2p+d,-}$. There exists a fourth root of unity $\omega$ such that the renormalized $S$-matrix of $\tilde{\mathcal{D}}_\xi$ is given by
  \[
    \mathbb{S}_{X(\lambda),X(\mu)} = \omega\xi^{\langle\lambda+\mu,(2p+d)\varpi_n\rangle+2pn(p+n-1)}\frac{\sum_{w\in W}(-1)^{l(w)}\xi^{2\langle w(\lambda+\rho),\mu+\rho\rangle}}{\sqrt{d}^n},
  \]
  for any $\lambda,\mu\in \tilde{C}_{n,d}$. The twist of the object $X(\lambda)$ is given by $\theta_{2p+d,X(\lambda)}=\xi^{\langle\lambda,\lambda+2\rho+(2p+d)\varpi_n\rangle}$.
\end{theorem}

\begin{proof}
  We need to compute the suitable renormalization of the $S$-matrix, which is given by the positive square root of the categorical (super)dimension of $\tilde{\mathcal{D}}_\xi$ multiplied by a square root of the quantum dimension of $\bar{\mathbf{1}}$.

  Thanks to \cref{lem:cat-dimension}, the categorical (super)dimension of $\tilde{\mathcal{D}}_\xi$ is given by
  \[
    \dim(\tilde{\mathcal{D}}_\xi) = (-1)^{n(n-1)/2}\frac{d^n}{\left(\prod_{i=1}^{n-1}(\xi^i-\xi^{-i})^{n-i}\right)^2}.
  \]
  In order to renormalize the $S$-matrix, we need the positive square root of this dimension. Using Weyl's denominator formula we obtain
  \[
    \sum_{w\in W}(-1)^{l(w)}\xi^{2\langle w(\rho),\rho\rangle}=\xi^{2\langle\rho,\rho\rangle}\prod_{i=1}^{n-1}(1-\xi^{-2i})^{n-i}=\xi^{\langle(n-1)\varpi_n,\rho\rangle}\prod_{i=1}^{n-1}(\xi^i-\xi^{-i})^{n-i},
  \]
  so that the desired square root is
  \[
    i^{n(n-1)/2}\frac{\xi^{\langle(n-1)\varpi_n,\rho\rangle}\sqrt{d}^n}{\displaystyle\sum_{w\in W}(-1)^{l(w)}\xi^{2\langle w(\rho),\rho\rangle}}.
  \]
  As the dimension of $\bar{\mathbf{1}}$ is $\pm\xi^{-2\langle (2p+n-1)\varpi_n,\rho+p\varpi_n\rangle}$, there exists a fourth root of unity $\omega$ such that
  \[
    \sqrt{\dim(\tilde{\mathcal{D}}_\xi)}\sqrt{\dim_{2p+d}(\bar{\mathbf{1}})} = \omega\frac{\sqrt{d}^n\xi^{-\langle (2p+n-1)\varpi_n,\rho+p\varpi_n\rangle+\langle(n-1)\varpi_n,\rho\rangle}}{\displaystyle\sum_{w\in W}(-1)^{l(w)}\xi^{2\langle w(\rho),\rho\rangle}}=\omega\frac{\sqrt{d}^n\xi^{-2pn(p+n-1)}}{\displaystyle\sum_{w\in W}(-1)^{l(w)}\xi^{2\langle w(\rho),\rho\rangle}}.
  \]

  Therefore, using \cref{prop:S-mat-gln} the renormalized $S$-matrix is given by
  \[
    \mathbb{S}_{X(\lambda),X(\mu)} = \omega\xi^{\langle\lambda+\mu,(2p+d)\varpi_n\rangle+2pn(p+n-1)}\frac{\sum_{w\in W}(-1)^{l(w)}\xi^{2\langle w(\lambda+\rho),\mu+\rho\rangle}}{\sqrt{d}^n}.
  \]
  The twist has already been computed in \cref{prop:S-mat-gln}.
\end{proof}

If $n\equiv d [2]$ then we have seen that the category $\mathcal{D}_{\xi}$ is slightly degenerate with symmetric center tensor generated by $\varepsilon$. This category inherits a $\mathbb{Z}/2d\mathbb{Z}$-grading from the $\mathbb{Z}$-grading of $\mathcal{C}_\xi$ as the object $\varepsilon^{\otimes 2}$ is of degree $2d$ in $\mathcal{C}_\xi$.

\begin{proposition}
  \label{prop:boxproduct}
  Suppose that $n\equiv d [2]$. The category $\mathcal{D}_\xi$ is equivalent to $\mathcal{D}_{\xi,0} \boxtimes \sVect$, where $\mathcal{D}_{\xi,0}$ is a non-degenerate braided fusion category, if and only if $n$ and $d$ are both odd or $n=d$.
\end{proposition}

\begin{proof}
  Suppose that $n$ and $d$ are both odd and consider the full subcategory $\mathcal{D}_{\xi,0}$ of $\mathcal{D}_\xi$ with objects of even degree in $\mathbb{Z}/2d\mathbb{Z}$. Since $d$ is odd, the object $\varepsilon$ of degree $d$ is not in $\mathcal{C}_\xi$. This shows that for each simple object $X$, either $X$ or $X\otimes \varepsilon$ is in $\mathcal{D}_{\xi,0}$ and that therefore $\mathcal{D}_\xi\simeq \mathcal{D}_{\xi,0} \boxtimes \sVect$.

  Suppose that $n=d$. Then the simple objects of $\mathcal{D}_\xi$ are $\mathbf{1}$ and $\varepsilon$ and then $\mathcal{D}_\xi\simeq \sVect$.

  Suppose now that both $n$ and $d$ are even, that $n>d$ and that there exists a full subcategory $\mathcal{D}_{\xi,0}$ of $\mathcal{D}_\xi$ such that $\mathcal{D}_\xi\simeq \mathcal{D}_{\xi,0} \boxtimes \sVect$. Since $d>n$, the object $X(\varpi_1)$ is a simple object of $\mathcal{D}_\xi$, and $X(\varpi_1)$ or $X(\varpi_1)\otimes \varepsilon$ is in $\mathcal{D}_{\xi,0}$. Because of \eqref{eq:tensor_fundamental}, the simple object $\varepsilon$ is a direct summand of $X(\varpi_1)^{\otimes d}$. Since $d$ is even, $X(\varpi_1)^{\otimes d}\simeq (X(\varpi_1)\otimes \varepsilon)^{\otimes d}$. Therefore $\varepsilon$ is in $\tilde{\mathcal{D}_\xi}$ which is a contradiction because the only simple transparent object of $\mathcal{D}_{\xi,0}$ is isomorphic to the unit object.
\end{proof}




\section{Exterior powers and Cuntz' positivity conjecture}
\label{sec:ext_powers}

Let $d\geq 1$ be an integer, $1 \leq n\leq d$ be another integer. We keep the notation $\xi=\exp(i\pi/d)$ and let $\zeta=\xi^2=\exp(2i\pi/d)$.

\subsection{Set-up and known results}
\label{sec:set-up_exterior}

We consider the $n$-th exterior power $\bigwedge^n S$ of the matrix $S=\left(\frac{\zeta^{ij}}{\sqrt{d}}\right)_{0\leq i,j < d}$, which is the renormalized character table $S$ of the cyclic group $\mathbb{Z}/d\mathbb{Z}$. We will index the entries of $\bigwedge^n S$ by the set $I_{n,d}$ of $n$-tuples $(i_1,\ldots,i_n)$ with $0\leq i_1 < i_2 < \cdots < i_n < d$. The matrix $\bigwedge^nS$ is symmetric and unitary. For $p\in\mathbb{Z}$ we denote by $i^{(p)}$ the $n$-tuple $(i_1^{(p)},\ldots,i_n^{(k)})$ obtained by reducing modulo $d$ and sorting the tuple $(p,p+1,\ldots,p+n-1)$. If $p\equiv p'\ [d]$ then $i^{(p)}=i^{(p')}$. Explicitly, if $0\leq p < d$,
\[
  i^{(p)} =
  \begin{cases}
    (p,p+1,\ldots,p+n-1) & \text{if } p \leq d-n,\\
    (0,1,\ldots,p+n-1-d,p,p+1,\ldots,d-1) & \text{otherwise}.
  \end{cases}
\]

The entry $(\bigwedge^nS)_{i^{(p)},i}$ is always non-zero since it is a multiple of a Vandermonde determinant. Therefore, we for any $a,b$ and $c$ ordered $n$-tuples, the following complex number:
\[
  {}_pN_{a,b}^c=\sum_{k\in I_{n,d}}\frac{(\bigwedge^nS)_{a,k}(\bigwedge^nS)_{b,k}\overline{(\bigwedge^nS)_{c,k}}}{(\bigwedge^nS)_{i^{(p)},k}}.
\]

\begin{proposition}[{\cite[Theorem 4.2]{cuntz-fusion}}]
  \label{prop:cuntz-integrality}
  For any $p\in\mathbb{Z}$ and $a,b,c \in I_{n,d}$, ${}_pN_{a,b}^c\in \mathbb{Z}$. Therefore, these integers are the structure constants of a $\mathbb{Z}$-algebra.
\end{proposition}

In \cite{cuntz-fusion}, it is proven only for $p=0$, but the same argument applies for any $p\in\mathbb{Z}$ (see \cite[Wahl der Eins]{cuntz-these}). Therefore, the matrix $\bigwedge^nS$ satisfies the hypothesis of \cref{sec:fusion_alg} for any choice of special element $i^{(p)}$ and defines a fusion algebra $A_{n,p}$. 

There also exist a $T$-matrix associated with $S$. Let $T$ be the diagonal matrix indexed by $I_{1,d}$ with entries given by $T_a=\zeta_{24}^{d-1}\xi^{a^2+da}$. Then $S$ and $T$ satisfy
\[
  S^4=\id,\quad (ST)^3 = \id\quad\text{and}\quad ST=TS,
\]
see \cite[Proposition 5.4]{cuntz-fusion} and similar relations are satisfied by $\bigwedge^nS$ and $\bigwedge^nT$.

\subsection{Cuntz' conjectures}
\label{sec:conj_positivity}

In \cite{cuntz-fusion}, Cuntz has conjectured the following positivity property:

\begin{conjecture}[{\cite[\S 4.3]{cuntz-fusion}}]
  \label{conj:cuntz_signs}
  Suppose that $1 < n < d$.
  \begin{enumerate}
  \item Suppose moreover that $n$ and $d$ are not both even. Then there exist a choice of signs $(\sigma_a)_{a\in I_{n,d}}\in\{\pm 1\}^{I_{n,d}}$ such that for any $a,b,c\in I_{n,d}$, the integer ${}_pN_{a,b}^c\sigma_a\sigma_b\sigma_c$ is non-negative.

  \item Suppose moreover that both $n$ and $d$ are even. Then for all choices of signs $(\sigma_a)_{a\in I_{n,d}}\in\{\pm 1\}^{I_{n,d}}$ there exists $a,b,c\in I_{n,d}$ such that ${}_pN_{a,b}^c\sigma_i\sigma_j\sigma_k$ is negative. However, the absolute values of ${}_pN_{a,b}^c$ define an associative $\mathbb{Z}$-algebra.
\end{enumerate}
\end{conjecture}

He also conjectured in his thesis \cite{cuntz-these} that the fusion ring $A$ defined by $\bigwedge^nS$ is a quotient of a free algebra of rank doubled.

\begin{conjecture}[{\cite[Vermutung 5.1.6]{cuntz-these}}]
  \label{conj:cuntz_ring}
  Let $A'$ be a free $\mathbb{Z}$-module with basis $\{b_a,b'_a\ \vert \ a \in I_{n,d}\}$ and denote also by $\{b_a\ \vert \ a \in I_{n,d}\}$ the basis of $A$. Let $\pi\colon A' \rightarrow A$ be the $\mathbb{Z}$-module map defined by $\pi(b_a) = b_a$ and $\pi(b_a')=-b_a$. Define also $\varphi \colon  A \rightarrow A'$ by
  \[
    \sum_{a\in I_{n,d}}\lambda_ab_a \mapsto \sum_{a\in I_{n,d}}(\delta_{\lambda_a>0}\lambda_ab_a - \delta_{\lambda_a<0}\lambda_ab'_a).
  \]
  Then the multiplication on $A'$ defined by $xy = \varphi(\pi(x)\pi(y))$ is associative and its structure constants lie in $\mathbb{N}$.
\end{conjecture}

The ring $A'$ has then two quotients, namely $A$ obtained by identifying $b_i$ and $-b'_i$ and another one $A^{\abs}$ obtained by identifying $b_i$ and $b'_i$. It is clear that $A^{\abs}$ has non-negative structure constants and that its structure constants are the absolute values of the structure constants of $A$. One can depicts the situation by the following diagram
\[
  \begin{tikzcd}
    & A' \ar[dl,two heads,"b_i=-b'_i"'] \ar[dr,two heads,"b_i=b'_i"] &\\
   A & & A^{\abs}
  \end{tikzcd}
\]
which is similar to the situation explained in \cref{sec:slightly-deg}.

Note that if one can find a change of basis of $A$ by changing signs such that the structure constants are non-negative, then \cref{conj:cuntz_ring} is almost trivial.

\subsection{Relationship with quantum $\mathfrak{gl}_n$}
\label{sec:relationship}

We now relate the renormalized $S$-matrix of the category $\tilde{\mathcal{D}}_\xi$ for $\gl_n$ introduced in \cref{sec:symmetric_center} with the exterior power $\bigwedge^nS$, up to some signs. The twists will also correspond with the diagonal matrix $\bigwedge^nT$, up to a multiplication by a root of unity. The pivotal structure on $\tilde{\mathcal{D}}_\xi$ will depend on the choice of the special element $i^{(p)}$.

We define for every $p\in\mathbb{Z}$ a map $\iota_p\colon I_{n,d} \rightarrow C_{n,d}$ by
\[
  a \mapsto w_0\left(\sum_{i=1}^na_{i}\varepsilon_i\right) - \rho - p\varpi_n.
\]
Since $a$ is strictly increasing, $\iota_p(a) \in P^+$. Moreover, $\iota_p(a)_1-\iota_p(a)_n = a_n-a_1-n+1 \leq d-n$, that is $\iota_p(a) \in C_{n,d}$. We then define $\tilde{\iota}_p(a) \in \tilde{C}_{n,d}$ as the unique element in $\tilde{C}_{n,d}$ such that $\tilde{\iota}_p(a)\sim \iota_p(a)$. As $\in_p$ is clearly an injection, so is $\tilde{\iota}_p$. But $\lvert \tilde{C}_{n,d} \rvert = \binom{d}{n} = \lvert I_{n,d} \rvert$ and $\tilde{\iota}_p$ is bijective.

\begin{lemma}
  \label{lem:invertible}
  For all $p\in \mathbb{Z}$ and $k\in\mathbb{Z}$ we have $\iota_p(i^{(k)}) \sim \left(k-p\right)\varpi_n$.
\end{lemma}

\begin{proof}
  We may and will suppose that $0 \leq k < d$ since $\lambda\sim \lambda+d\varpi_n$ for any $\lambda\in C_{n,d}$.

  We check that
  \[
    \iota_p(i^{(k)})=
    \begin{cases}
      \left(k-p\right)\varpi_n & \text{if } 0 \leq k \leq d-n,\\
      \displaystyle -p\varpi_n + (d-n)\varpi_{d-k}& \text{otherwise}.
    \end{cases}
  \]
  If $d-n < k < d$, we see that
  \[
    \iota_p(i^{(k)}) = \sh^{d-k}\left((k-d-p)\varpi_n\right) + (d-n)\varpi_{d-k} + (d-k)\varpi_n\sim \left(k-d-p\right)\varpi_n\sim \left(k-p\right)\varpi_n,
  \]
  which ends the proof.
\end{proof}

Therefore $\tilde{\iota}_p(i^{(p)}) = 0$ and $X(\tilde{\iota}_p(i^{(-p+1-n)}))\simeq \bar{\mathbf{1}}$.

\begin{theorem}
  \label{thm:categorification}
  Let $p\in\mathbb{Z}$. We equip the category $\tilde{\mathcal{D}}_\xi$ with the pivotal structure $a_{2p+d,-}$. The (super)fusion category $\tilde{\mathcal{D}}_\xi$ is a categorification of the modular datum defined by $\bigwedge^nS$ and $\bigwedge^nT$: there exist a fourth root of unity $\omega$ and signs $(\sigma_a)\in\{\pm 1\}^{I_{n,d}}$ with $\sigma_{i^{(p)}}=1$ such that
\[
  \mathbb{S}_{X(\tilde{\iota}_p(a)),X(\tilde{\iota}_p(b))} = \omega\sigma_a\sigma_b(\bigwedge\nolimits^nS)_{a,b} \quad\text{and}\quad \theta_{X(\tilde{\iota}_p(a))} = \zeta_*(\bigwedge\nolimits^nT)_{a},
\]
where $\zeta_* = \zeta_{24}^{n(1-d)}\xi^{-\langle\rho,\rho\rangle-pn(p+d)-(2p+d)\binom{n}{2}}$.

The (super)Grothendieck ring of $\tilde{\mathcal{D}}_\xi$ is then isomorphic to the ring defined by $\bigwedge^nS$ with unit parameterized by $i^{(p)}$.
\end{theorem}

\begin{proof}
  We use the formula of the renormalized $S$-matrix of $\tilde{\mathcal{D}}_\xi$ given in \cref{thm:modular_data_gln}. Since for all $a\in C_{n,d}$ we have $\tilde{\iota}_p(a)\sim \iota_p(a)$, there exist a sign $\eta_a\in\{\pm 1\}$ such that for all $a,b\in C_{n,d}$,
  \[
    \mathbb{S}_{X(\tilde{\iota}_p(a)),X(\tilde{\iota}_p(b))} = \omega\eta_a\eta_b\xi^{\langle\iota_p(a)+\iota_p(b),(2p+d)\varpi_n\rangle+2pn(p+n-1)}\frac{\sum_{w\in W}(-1)^{l(w)}\xi^{2\langle w(\iota_p(a)+\rho),\iota_p(b)+\rho\rangle}}{\sqrt{d}^n}.
  \]

  Now, for any $w\in W$, we have
  \begin{align*}
    \langle w(\iota_p(a)+\rho),\iota_p(b)+\rho\rangle
    &= \langle ww_0(\sum_{i=1}^na_i\varepsilon_i)-p\varpi_n,w_0(\sum_{i=1}^nb_i\varepsilon_i)-p\varpi_n\rangle \\
    &= \sum_{i=1}^na_{w_0^{-1}w^{-1}w_0(i)}b_i - p\sum_{i=1}^n(a_i+b_i) + p^2n,
  \end{align*}
  so that
  \[
    \mathbb{S}_{X(\tilde{\iota}_p(a)),X(\tilde{\iota}_p(b))}=\omega\eta_a\eta_b(-1)^{\sum_{i=1}^n(a_i+b_i)}\frac{\sum_{w\in W}(-1)^{l(w)}\prod_{i=1}^n\xi^{a_{w(i)}b_i}}{\sqrt{d}^n}.
  \]

  We now set $\sigma_a=\eta_a(-1)^{\sum_{j=1}^n a_j}$. We hence find that $\mathbb{S}_{X(\tilde{\iota}_p(a)),X(\tilde{\iota}_p(b))}=\omega \sigma_a\sigma_b (\bigwedge^nS)_{a,b}$. By changing the sign of every $\sigma_a$ if necessary, we have $\sigma_{i^{(p)}}=1$.

  For the value of the twist, since $\theta_{2p+d,X(\tilde{\iota}_p(a))}=\theta_{2p+d,X(\iota_p(a))}$, we find that
  \[
    \theta_{2p+d,X(\tilde{\iota}_p(a))} = \xi^{\langle\iota_p(a),\iota_p(a)+2\rho+(2p+d)\varpi_n\rangle}.
  \]
  But 
  \[
    \langle\iota_p(a),\iota_p(a)+2\rho+(2p+d)\varpi_n\rangle = \langle w_0(\sum_{i=1}^n a_i\varepsilon_i)-\rho-p\varpi_n,w_0(\sum_{i=1}^n a_i\varepsilon_i)+\rho+p\varpi_n+d\varpi_n\rangle,
  \]
  and therefore $\theta_{2p+d,X(\tilde{\iota}_p(a))} = \xi^{-\langle\rho,\rho\rangle-pn(p+d)-(2p+d)\binom{n}{2}}\xi^{\sum_{i=1}^n(a_i^2+da_i)}$, which leads to the desired formula.
\end{proof}

As a corollary, we obtain a new proof of \cref{prop:cuntz-integrality} since the signs $\sigma_a$ do not change the integrality of the structure constants ${}_pN_{a,b}^c$.

\begin{corollary}
  For any $p$ and $r$, the fusion algebras $A_{n,p}$ and $A_{n,r}$ are isomorphic: the structure of the fusion ring defined bu $\bigwedge^n S$ does not depend on the choice of the special element of the form $i^{(p)}$.
\end{corollary}

\subsection{Proof of Cuntz' conjectures}

Categorification has turned to be a powerful tool to prove positivity conjectures, and the categorical interpretation of the matrix $\bigwedge^n S$ in terms of the category $\tilde{\mathcal{D}}_\xi$ will be crucial for proving the conjectures.

\begin{theorem}
  \label{thm:conj_cuntz_signs}
  \cref{conj:cuntz_signs} is true.
\end{theorem}

\begin{proof}
  We start first with the case $n\not\equiv d\ [2]$. The category $\tilde{\mathcal{D}}_\xi$ equipped with the pivotal structure $a_{2p+d,-}$ is a non-degenerate fusion category with simple objects $X(\lambda)$ for $\lambda\in \tilde{C}_{n,d}$. The Verlinde formula asserts that the multiplicity of $X(\gamma)$ in $X(\alpha)\otimes X(\beta)$ is given by
\[
  \sum_{\kappa \in \tilde{C}_{n,d}}\frac{\mathbb{S}_{\alpha,\kappa}\mathbb{S}_{\beta,\kappa}\bar{\mathbb{S}}_{\gamma,\kappa}}{\mathbb{S}_{0,\kappa}} = {}_p N_{a,b}^c\sigma_a\sigma_b \sigma_c,
\]
where $\tilde{\iota}_p(a)=\alpha, \tilde{\iota}_p(b) = \beta$, and $\tilde{\iota}_p(c)=\gamma$, the equality following from \cref{thm:categorification}. Since a multiplicity in a fusion category is non-negative, we deduce that ${}_p N_{a,b}^c\sigma_a\sigma_b \sigma_c$ is non-negative for all $a,b,c\in I_{n,d}$.

We now turn to the case of $n$ and $d$ odd. Since the category $\tilde{\mathcal{D}}_\xi$ is a non-degenerate superfusion category, the same argument only prove that the structure constants are integers. But we have seen in \cref{prop:boxproduct} that in this case, the category $\mathcal{D}_\xi$, whose $\tilde{\mathcal{D}}_\xi$ is a supermodularization, is equivalent to $\mathcal{D}_{\xi,0}\boxtimes \sVect$ with $\mathcal{D}_{\xi,0}$ a non-degenerate braided fusion category. The modular invariants of $\mathcal{D}_{\xi,0}$ and $\tilde{\mathcal{D}}_\xi$ coincide up to signs, and the end of the proof is similar to the case $n\not\equiv d\ [2]$.

Finally, if both $n$ and $d$ are even, the category $\mathcal{D}_\xi$ is slightly degenerate and since $n>d$, \cref{prop:boxproduct} asserts that the category $\mathcal{D}_\xi$ is not of the form $\mathcal{D}_{\xi,0}\boxtimes \sVect$ for $\mathcal{D}_{\xi,0}$ non-degenerate. Hence we cannot find a change of signs $(\sigma_a)_{a\in I_{n,d}}\in\{\pm 1\}^{I_{n,d}}$ such that for any $a,b,c\in I_{n,d}$, the integer ${}_pN_{a,b}^c\sigma_a\sigma_b\sigma_c$ is non-negative.
\end{proof}

\begin{theorem}
  \label{thm:conj_cuntz_ring}
  \cref{conj:cuntz_ring} is true.
\end{theorem}

\begin{proof}
  Only the case of $n$ and $d$ even needs a comment. Since the fusion ring defined by $\bigwedge^nS$ is isomorphic to the quotient of the Grothendieck ring $\Gr(\mathcal{D}_{\xi})$ by the ideal generated by $[\varepsilon]+[\mathbf{1}]$, we easily check that the ring $\Gr(\mathcal{D}_{\xi})$ is isomorphic to the ring $A'$ of \cref{conj:cuntz_ring}. Note that we crucially need that $N_{X,Y}^ZN_{X,Y}^{Z\otimes \varepsilon} = 0$ in $\Gr(\mathcal{D}_{\xi})$, which is true since $\varepsilon$ sits in non trivial degree for the $\mathbb{Z}/2d\mathbb{Z}$-grading.
\end{proof}



\section{Fourier matrices for $G(d,1,n)$}
\label{sec:fourier_d1n}

In \cite{cuntz-fusion}, Cuntz noticed that the Fourier matrices defined by Malle \cite{unipotente} can be expressed using tensor products of exterior powers $\bigwedge^nS$. Since we constructed a categorification of these exterior powers using representations of the quantum enveloping algebra of $\gl_n$ at an even root of unity, we now explore the categorification of the Fourier matrices for $G(d,1,n)$. We fix $d\geq 1$ an integer.

Since we will simultaneously consider different values of $n$, we will add a subscript $n$ to the various objects considered in the previous sections. For example, we will denote the $\mathfrak{gl}_n$-weight $\rho$ by $\rho_n$, the Weyl group $W$ of $\mathfrak{gl}_n$ by $W_n$ and so on.

\subsection{Fourier matrices and exterior powers}
\label{sec:fourier_ext}

We follow the presentation of \cite[Section 3]{cuntz-fusion}. Let $m\in \mathbb{N}$, $Y$ be a totally ordered set with $md+1$ elements and $\pi \colon Y \rightarrow \mathbb{N}$ be a map. Let $w_1<\ldots < w_r$ be such that $\pi(Y)=\{w_1,\ldots,w_r\}$ and $n_i=\lvert \pi^{-1}(w_i) \rvert$. Then we have $\sum_{i=1}^rn_i =md+1$.

We consider the set $\Psi$ of maps $f \colon Y \rightarrow \{0,\ldots,d-1\}$ such that $f$ is strictly increasing on $\pi^{-1}(i)$ for each $i\in\mathbb{N}$. If $f\in \Psi$, we define a sign $\varepsilon(f)\in\{\pm 1\}$ by
\[
  \varepsilon(f) = (-1)^{\left\lvert\{ (y,y') \in Y\times Y\ \middle\vert\ y<y'\text{ and } f(y)<f(y')\}\right\rvert}.
\]

We consider the subset $\Xi$ of functions $f\in \Psi$ such that $\sum_{y\in Y}f(y) \equiv m\binom{d}{2}\ [d]$ and remark that $\Xi$ has $\frac{1}{d}\prod_{i=1}^r\binom{d}{n_i}$ elements. Such functions can be interpreted as symbols in the sense of \cite{unipotente}, which parameterize unipotent characters for the complex reflection group $G(d,1,n)$. For such a function $f\in\Xi$, we denote by $f_i=(f_{i,1},\ldots,f_{i,n_i})$ the ordered tuple obtained from $f(\pi^{-1}(w_i))$. The datum of the function $f$ is equivalent to the data of $f_1,\ldots, f_r$. The Fourier matrix is defined as the matrix indexed by $\Xi$ with entries
\[
  \mathcal{S}_{f,g} = (-1)^{m(d-1)}i^{-\binom{d-1}{2}m}\sqrt{d}\varepsilon(f)\varepsilon(g)\prod_{i=1}^r\overline{\left(\bigwedge\nolimits^{n_i}S\right)_{f_i,g_i}},
\]
and to each function $f\in \Xi$, Malle also associate an eigenvalue of the Frobenius:
\[
  \Fr(f) = \zeta_{12d}^{md(1-d^2)}\prod_{y\in Y}\zeta_{12d}^{-6(f(y)^2+df(y))},
\]
where $\zeta_{12d}^d = \zeta$. We denote by $\mathcal{T}$ the diagonal matrix indexed by $\Xi$ with entries given by the eigenvalues of the Frobenius. Note that, up to a scalar, the eigenvalue of the Frobenius $\Fr(f)$ coincides with $\prod_{i=1}^r\overline{(\bigwedge^{n_i} T)_{f_i}}$.

\begin{proposition}[{\cite[4.15]{unipotente},\cite[Proposition 5.1]{cuntz-fusion}}]
  \label{prop:ST_d1n}
  Let $Y$, $\pi$ and $\Xi$ as above and the associated matrices $\mathcal{S}$ and $\mathcal{T}$.
  \begin{enumerate}
  \item The matrix $\mathcal{S}$ is symmetric and unitary.
  \item The matrices $\mathcal{S}$ and $\mathcal{T}$ satisfy
    \[
      \mathcal{S}^4=\id,\quad (\mathcal{S}\mathcal{T})^3=\id\quad\text{and}\quad \mathcal{S}^2\mathcal{T}=\mathcal{T}\mathcal{S}^2.
    \]
  \item For $1\leq i \leq r$ choose $p_i$ such that $f_0\in \Psi$ defined by $(f_0)_{i}=i^{(p_i)}$ is in $\Xi$. Then the structure constants
    \[
      {}_{f_0}N_{f,g}^h = \sum_{k\in \Xi}\frac{\mathcal{S}_{f,k}\mathcal{S}_{g,k}\overline{\mathcal{S}_{h,k}}}{\mathcal{S}_{f_0,k}}
    \]
    are integers for every $f,g$ and $h\in \Xi$. 
  \end{enumerate}
\end{proposition} 

\subsection{Symmetric center of some Deligne tensor products}
\label{sec:cat_fourier_d1n}

Recall the that the category $\mathcal{D}_{n,\xi^{-1}}$ of \cref{sec:symmetric_center} is obtained from $\mathcal{C}_{n,\xi^{-1}}$ by adding isomorphisms and that the symmetric center of $\mathcal{D}_{n,\xi^{-1}}$ is tensor generated by $\varepsilon_n = X((d-n)\omega_1+\varpi_n)$ which satisfies $\varepsilon\otimes \varepsilon \simeq \mathbf{1}$ in $\mathcal{D}_{n,\xi^{-1}}$. We consider a Deligne tensor product of the categories $\mathcal{D}_{n,\xi^{-1}}$ associated with the representations the quantum enveloping algebra of $\gl_n$ at a root of unity.

Let $\mathcal{D}_{\underline{n},\xi^{-1}}=\mathcal{D}_{n_1,\xi^{-1}} \boxtimes \ldots \boxtimes \mathcal{D}_{n_r,\xi^{-1}}$ for $\underline{n} = (n_1,\ldots,n_r)$ such that $1\leq n_i \leq d$ and $\sum_{i=1}^nn_i \equiv 1 [d]$. This category is a braided fusion category admitting many pivotal structures. Its simple objects are of the form $X(\underline{\lambda}) = X(\lambda_1)\boxtimes \cdots \boxtimes X(\lambda_r)$ with $\lambda_i \in C_{n_i,d}$. Its symmetric center has $2^r$ simple objects given by $\varepsilon_{\underline{\delta}} = \varepsilon_{n_1}^{\otimes \delta_1}\boxtimes \cdots \boxtimes \varepsilon_{n_r}^{\otimes \delta_r}$, where $\underline{\delta} = (\delta_1,\ldots,\delta_r) \in \{0,1\}^r$ and $\varepsilon_{n_i}$ is the unique simple transparent object of $\mathcal{D}_{n_i,\xi^{-1}}$ non-isomorphic to the unit object. The category $\mathcal{D}_{\underline{n},\xi^{-1}}$ has a natural $(\mathbb{Z}/2d\mathbb{Z})^r$-grading: a simple object $X(\underline{\lambda})$ sits in degree $(\langle\lambda_i,\varpi_{n_i}\rangle)_{1\leq i \leq r}$.

Consider now $\mathcal{E}_{\underline{n},\xi^{-1}}$ the full subcategory with simple objects $X(\underline{\lambda})$ satisfying $\sum_{i=1}^r \langle\lambda_i,\varpi_{n_i}\rangle \equiv 0 [d]$. Thanks to the grading, it is easily seen that the category $\mathcal{E}_{\underline{n},\xi^{-1}}$ is stable under the tensor product and is thus a braided fusion category.

From the construction of the fusion datum of \cref{sec:fourier_ext} and the results of \cref{sec:ext_powers}, we expect that the category $\mathcal{E}_{\underline{n},\xi^{-1}}$ gives a categorification of the fusion datum of \cref{sec:fourier_ext}. We first determine the symmetric center of $\mathcal{E}_{\underline{n},\xi^{-1}}$ in order to ensure that its $S$-matrix has rank $\frac{1}{d}\prod_{i=1}^r\binom{d}{n_i}$.

Since $\varepsilon_i$ is of degree $d$, the objects $\varepsilon_{\underline{\delta}}$ for $\underline{\delta}\in\{0,1\}^d$ are in the symmetric center of $\mathcal{E}_{\underline{n},\xi^{-1}}$, and there are no other simple transparent objects:

\begin{proposition}
  The symmetric center of $\mathcal{E}_{\underline{n},\xi^{-1}}$ is the full subcategory with simple objects $\varepsilon_{\underline{\delta}}$ for $\underline{\delta}\in\{0,1\}^d$.
\end{proposition}

\begin{proof}
  We use the same strategy as in the proof of \cref{prop:symmetric_center}. The same proof shows that if $X$ is a transparent simple object, then for any simple object $Y$ the morphism $\theta_{X\otimes Y}$ is a scalar multiple of the identity, where $\theta$ is the pivotal structure obtained from the tensor product of the pivotal structures $\theta_{0,-}$. We now fix such a transparent simple object $X(\underline{\lambda})$ and consider for $1 \leq i \leq r$ the simple object $Y_i=X(\underline{\mu}^{(i)})$ where
  \[
    \mu_j^{(i)} =
    \begin{cases}
      -\varpi_{n_j} & \text{ if } j\neq i,\\
      \varpi_1 - \varpi_{n_j}& \text{ if } j=i.
    \end{cases}
  \]\
  Since $\sum_{j=1}^rn_j \equiv 1 [d]$, the object $Y_i$ is in $\mathcal{E}_{\underline{n},\xi^{-1}}$. As in the proof of \cref{prop:symmetric_center}, considering decomposition of the tensor product $X(\underline{\lambda})\otimes Y_i$ shows that $X(\lambda_i)$ is invertible and therefore $X(\underline{\lambda})$ is also invertible.
  
  We hence may and will suppose that for every $1 \leq i \leq r$, there exists $1 \leq l_i \leq n_i$ and $s_i \in \mathbb{Z}$ such that $\lambda_i = (d-n_i)\varpi_{l_i} + s_i \varpi_{n_i}$. We aim to show that $l_i\equiv s_i [d]$ for every $1 \leq i \leq r$. Then $X(\underline{\lambda})$ is transparent if and only if $\theta_{X(\underline{\lambda})\otimes Y(\underline{\mu})} = \theta_{X(\underline{\lambda})} \theta_{Y(\underline{\mu})}$ for any $\underline{\mu}$ which is equivalent to
  \begin{multline*}
    \sum_{j=1}^r\langle\sh^{l_j}(\mu_j) + (d-n_j)\varpi_{l_j} + s_j\varpi_{n_j},\sh^{l_j}(\mu_j) + (d-n_j)\varpi_{l_j} + s_j\varpi_{n_j}+2\rho\rangle \equiv \\
   \sum_{j=1}^r\langle\mu_j,\mu_j+2\rho\rangle+\langle(d-n_j)\varpi_{l_j} + s_j\varpi_{n_j},(d-n_j)\varpi_{l_j} + s_j\varpi_{n_j}+2\rho\rangle\ [2d].
 \end{multline*}
 As in the proof of \cref{prop:symmetric_center}, one may show that this is  the equivalent to
 \[
   \sum_{j=1}^r (s_j-l_j)\langle\mu_j,\varpi_{n_j}\rangle \equiv 0 [d].
 \]
 We once again choose $\underline{\mu}=\underline{\mu}^{(i)}$ so that $\langle\mu_j^{(i)},\varpi_{n_j}\rangle = \delta_{i,j}-n_j$. Hence if $X(\underline{\lambda})$ is transparent, we obtain that $\sum_{j=1}^r n_j(l_j-s_j) + (s_i-l_i) \equiv 0 [d]$. But as $X(\underline{\lambda})$ is in $\mathcal{E}_{\underline{n},\xi^{-1}}$, we have $\sum_{j=1}^rn_j(s_j-l_j) \equiv 0 [d]$ and hence $(s_i-l_i) \equiv 0 [d]$ for every $1 \leq i \leq r$. This shows that $X(\underline{\lambda})$ is isomorphic to an object of the form $\varepsilon_{\underline{\delta}}$.
\end{proof}

It is readily seen that $\varepsilon_{\underline{\delta}}$ is in non trivial degree if there exist $1\leq i \leq r$ such that $\delta_i=1$. Therefore tensoring by $\varepsilon_{\underline{\delta}}\not\simeq \mathbf{1}$ has no fixed points on the set of simple objects. As in \cref{sec:symmetric_center}, we want to (super)modularize the category $\mathcal{E}_{\underline{n},\xi^{-1}}$ and therefore we need that $\varepsilon_{\underline{\delta}}$ has a twist equal to $1$ for the chosen pivotal structure and every $\underline{\delta}\in\{0,1\}^r$. We thus equip $\mathcal{D}_{n_i,\xi^{-1}}$ with a pivotal structure of the form $a_{2p_i+d,-}$ for some $p_i\in \mathbb{Z}$. The category $\mathcal{E}_{\underline{n},\xi^{-1}}$ is then equipped with a pivotal structure that we will denote by $a_{2\underline{p}+d,-}$; the corresponding twist will be denoted by $\theta_{2\underline{p}+d,-}$.

\begin{corollary}
  The symmetric center of $\mathcal{E}_{\underline{n},\xi^{-1}}$ equipped with the pivotal structure $a_{2\underline{p}+d,-}$ is tensor generated by the objects $\varepsilon_{\underline{\delta}}$ for $\delta\in \{0,1\}^r$. Moreover, $\varepsilon_{\underline{\delta}}$ is of quantum dimension $\prod_{\substack{1 \leq i \leq r\\ n_i\equiv d [2]}}(-1)^{\delta_i}$ and of twist $1$.

  The (super)modularization $\tilde{\mathcal{E}}_{\underline{n},\xi^{-1}}$  of $\mathcal{E}_{\underline{n},\xi^{-1}}$ has its objects parameterized by the set
  \[
    \tilde{E}_{\underline{n},d}=\left\{\underline{\lambda} \in \tilde{C}_{n_1,d}\times \cdots \times \tilde{C}_{n_r,d}\ \middle\vert\ \sum_{i=1}^r\langle\lambda_i,\varpi_{n_i}\rangle \equiv 0 [d]\right\}.
  \]
\end{corollary}

Note that the resulting category $\tilde{\mathcal{E}}_{\underline{n},\xi^{-1}}$ is a superfusion category as soon there exists $i$ such that $n_i$ and $d$ have the same parity. We summarize the results on the category $\tilde{\mathcal{E}}_{\underline{n},\xi^{-1}}$ and its modular invariants.

\begin{proposition}
  \label{prop:modular_inv_d1n}
  Recall that we endow the category the category $\tilde{\mathcal{E}}_{\underline{n},\xi^{-1}}$ with the pivotal structure $a_{2\underline{p}+d,-}$. Then there exists a fourth root of unity $\omega$ such that the renormalized $S$-matrix of $\tilde{\mathcal{E}}_{\underline{n},\xi^{-1}}$ is given by
  \[
    \mathbb{S}_{\underline{\lambda},\underline{\mu}} = \frac{\omega}{\sqrt{d}}\prod_{i=1}^r\xi^{-\langle\lambda_i+\mu_i,(2p_i+d)\varpi_{n_i}\rangle-2p_in_i(p_i+n_i-1)}\frac{\sum_{w\in W_{n_i}}(-1)^{l(w)}\xi^{-2\langle w(\lambda_i+\rho_{n_i}),\mu_i+\rho_{n_i}\rangle}}{\sqrt{d}^{n_i}}
  \]
  for any $\underline{\lambda},\underline{\mu}\in \tilde{E}_{\underline{n},d}$. The twist on the simple object $X(\underline{\lambda})$ is given by multiplication by $\xi^{-\sum_{i=1}^r\langle\lambda_i,\lambda_i+2\rho+(2p_i+d)\varpi_{n_i}\rangle}$.
\end{proposition}

\begin{proof}
  This follows immediately from \cref{thm:modular_data_gln}. The extra $\sqrt{d}$ at the denominator comes from the fact we work with the modularization of the subcategory $\mathcal{E}_{\underline{n},\xi^{-1}}$ and not with the modularization of the whole category $\mathcal{D}_{\underline{n},\xi^{-1}}$. Indeed, $\dim(\mathcal{E}_{\underline{n},\xi^{-1}}) = \dim(\mathcal{D}_{\underline{n},\xi^{-1}})/d$ thanks to the grading.
\end{proof}

\subsection{Fourier matrix and eigenvalues of the Frobenius as modular invariants}
\label{sec:fourier_modular_inv_d1n}

Finally, as expected, we recover the Fourier matrix $\mathcal{S}$ and the eigenvalues of the Frobenius $\mathcal{T}$ of \cref{sec:fourier_ext} from the category $\tilde{\mathcal{E}}_{\underline{n},\xi^{-1}}$. We choose the integers $p_1,\ldots,p_r$ as in \cref{prop:ST_d1n}, that is such that $f_0$ defined by $(f_0)_i =i^{(p_i)}\in I_{n_i,d}$ is in $\Xi$. This condition amounts to
\[
  \sum_{i=1}^r \left(p_in_i + \binom{n_i}{2}\right) \equiv m\binom{d}{2} [d].
\]

Using the various maps $\tilde{\iota}_p$ from \cref{sec:relationship}, we define a map $\tilde{\iota}_{\underline{n},\underline{p}} \colon \Xi \rightarrow \tilde{E}_{\underline{n},d}$ by
\[
  \tilde{\iota}_{\underline{n},\underline{p}}(f)_i = \tilde{\iota}_{n_i,p_i}(f_i),
\]
where $f_i\in I_{n_i,d}$ is as in \cref{sec:fourier_ext}. Note that $\tilde{\iota}_{\underline{n},\underline{p}}(f)$ is indeed in $\tilde{E}_{\underline{n},d}$ since
\[
  \sum_{i=1}^r\langle\tilde{\iota}_{n_i,p_i}(f_i),\varpi_n\rangle = \sum_{y\in Y}f(y) - \sum_{i=1}^r \left(p_in_i + \binom{n_i}{2}\right) \equiv 0 [d],
\]
the last equality following from the fact that $f\in \Xi$ and that the integers $p_1,\ldots,p_r$ are chosen such that $f_0 \in \Xi$. The map $\tilde{\iota}_{\underline{n},\underline{p}}$ is bijective since for every $p$ the map $\tilde{\iota}_{p}$ is bijective and that $\lvert \Xi \rvert = \frac{1}{d}\prod_{i=1}^r\binom{d}{n_i} = \lvert\tilde{E}_{\underline{n},d}\rvert$.

\begin{theorem}
  \label{thm:categorification_d1n}
  We keep the above notations. The (super)category of $\tilde{\mathcal{E}}_{\underline{n},\xi^{-1}}$ is a categorification of the modular datum defined by $\mathcal{S}$ and $\mathcal{T}$: there exist a fourth root of unity $\omega$ and signs $(\sigma_f)\in\{\pm 1\}^{\Xi}$ with $\sigma_{f_0}=1$ such that
\[
  \mathbb{S}_{X(\tilde{\iota}_{\underline{n},\underline{p}}(f)),X(\tilde{\iota}_{\underline{n},\underline{p}}(g))} = \omega\sigma_f\sigma_g\mathcal{S}_{f,g} \quad\text{and}\quad \theta_{d+2\underline{p},X(\tilde{\iota}_p(a))}= \Fr(f_0)^{-1}\mathcal{T}_f.
\]

The (super)Grothendieck ring of $\tilde{\mathcal{E}}_{\underline{n},\xi^{-1}}$ is then isomorphic to the ring defined by $\mathcal{S}$ with unit parameterized by $f_0$.
\end{theorem}

\begin{proof}
  The proof is similar to the one of \cref{thm:categorification} using the modular invariants of $\tilde{\mathcal{E}}_{\underline{n},\xi^{-1}}$ given in \cref{prop:modular_inv_d1n}, and is then omitted.
\end{proof}

As a corollary, we obtain an independent proof of the integrality of the structure constants defined by the matrix $\mathcal{S}$, and moreover that the absolute value of these structure constants also define an associative ring.

Finally, note that if we choose for $f_0$ the special symbol as in \cite[Bemerkung 2.25]{unipotente}, then it moreover satisfies $\Fr(f_0) = 1$ and the eigenvalues of the Frobenius coincide with the twist in $\tilde{E}_{\underline{n},\xi^{-1}}$.

\subsection{Ennola $d$-ality}
\label{sec:ennola}

Given an element $f\in \Xi$, Malle has defined a polynomial $\gamma_f(q)$ which behavior is similar to the degrees of the unipotent characters of a finite group of Lie type. In particular, a property similar to the Ennola duality exists, but ii is rather a $d$-ality. There exists a bijection $E \colon \Xi \rightarrow \Xi$ such that, up to a sign, the polynomials $\gamma_{f}(\zeta q)$ and $\gamma_{E(f)}(q)$ coincide up to a sign. Therefore $E^d$ is the identity. This bijection is defined explicitly by Malle \cite[Folgerung 3.11]{unipotente} in terms of $d$-symbols, and we give the translation in terms of functions in $\Xi$. Let $f\in \Xi$. Its Ennola transform is the unique function $E(f)\in \Xi$ such that $E(f)_i$ is given by reducing modulo $0$ and sorting increasingly the set $f(\pi^{-1}(w_i))+w_i-\sum_{k=1}^r w_kn_k$. It is easily checked that $E(f)$ indeed belongs to $\Xi$.

\begin{proposition}
  Let $\eta$ be the invertible object $X(\underline{\eta})$ with $\eta_i = (w_i-\sum_{k=1}^r w_kn_k)\varpi_{n_i}$. Then the objects $X(\tilde{\iota}_{\underline{n},\underline{p}}(f))\otimes \eta$ and $X(\tilde{\iota}_{\underline{n},\underline{p}}(E(f)))$ are isomorphic in $\mathcal{E}_{\underline{n},\xi^{-1}}$: the Ennola $d$-ality is given by tensoring by an invertible object of trivial $d$-th tensor power.
\end{proposition}

\begin{proof}
  It is clear that $\eta$ is indeed an object in $\mathcal{E}_{\underline{n},\xi^{-1}}$, that is that $\sum_{i=1}^r\langle\eta_i,\varpi_{n_i}\rangle \equiv 0 [d]$.
  
  As $X(\eta_i)$ is a tensor power of $\det_{n_i,\xi^{-1}}$, it suffices to show that for every $1 \leq i \leq r$ we have $\iota_{\underline{n},\underline{p}}(f)+\eta_i \sim \iota_{\underline{n},\underline{p}}(E(f))$, which is immediate by definition of $\eta$.
\end{proof}





\bibliographystyle{bibliography/habbrv}
\bibliography{bibliography/biblio}


\end{document}